\documentclass[a4paper,11pt,fleqn]{article}
\usepackage{amsmath,amssymb,amsthm}
\usepackage{mathrsfs} 
\usepackage{euscript}
\usepackage{color}
\usepackage[shortlabels]{enumitem}
\usepackage{bm}
\usepackage{charter}
\usepackage{comment}
\usepackage{srcltx}
\usepackage{xcolor}
\usepackage{graphicx}
\usepackage{amsmath}
\usepackage{amssymb}
\usepackage{euscript}
\usepackage{epic,eepic}
\usepackage{pstricks}
\usepackage{parskip}

\usepackage{tikz}
\usepackage{color}
\usepackage{multirow}
\usepackage{float}
\usepackage{multirow, array}
\usepackage{booktabs}
\usepackage{mathrsfs}
\usepackage{mathdots}
\usepackage{stackrel}
\usepackage{relsize}
\usepackage{ dsfont }
\usepackage{listings}
\usepackage{color}
\usepackage{graphicx}
\usepackage{setspace} 
\usepackage[utf8]{inputenc}
\usepackage{xcolor}
\usepackage[USenglish]{babel}
\usepackage{mathtools}
\usepackage{subfigure}
\usepackage{amsmath}
\usepackage{amssymb}
\usepackage{euscript}
\usepackage{epic,eepic}
\usepackage{pstricks}
\usepackage{parskip}
\usepackage{tikz}
\usepackage{multirow}
\usepackage{float}
\usepackage{array}
\usepackage{booktabs}
\usepackage{mathrsfs}
\usepackage{tabularx}
\usepackage{mathdots}
\usepackage{stackrel}
\usepackage{fancybox, calc}
\usepackage{verbatim}
\usepackage{fancyhdr}
\usepackage{tabularx}

\definecolor{mygreen}{rgb}{0,0.6,0}
\definecolor{mygray}{rgb}{0.5,0.5,0.5}
\definecolor{mymauve}{rgb}{0.58,0,0.82}
\definecolor{labelkey}{rgb}{0,0.08,0.45}
\definecolor{refkey}{rgb}{0,0.6,0.0}
\definecolor{Brown}{rgb}{0.45,0.0,0.05}
\definecolor{dgreen}{rgb}{0.00,0.49,0.00}
\definecolor{dblue}{rgb}{0,0.08,0.75}
\DeclareSymbolFont{largesymbolsA}{U}{jkpexa}{m}{n}
\SetSymbolFont{largesymbolsA}{bold}{U}{jkpexa}{bx}{n}
\DeclareMathSymbol{\varprod}{\mathop}{largesymbolsA}{16}

\definecolor{labelkey}{rgb}{0,0.08,0.45}
\definecolor{refkey}{rgb}{0,0.6,0.0}
\definecolor{Brown}{rgb}{0.45,0.0,0.05}
\definecolor{dgreen}{rgb}{0.00,0.49,0.00}
\definecolor{dblue}{rgb}{0,0.08,0.75}
\RequirePackage[colorlinks,hyperindex]{hyperref} 
\hypersetup{linktocpage=true,citecolor=dblue,linkcolor=dgreen}

\newtheorem{theorem}{Theorem}[section]
\newtheorem{corollary}[theorem]{Corollary}
\newtheorem{lemma}[theorem]{Lemma}

\newtheorem{assumption}[theorem]{Assumption}
\theoremstyle{definition}
\newtheorem{definition}[theorem]{Definition}

\newtheorem{remark}[theorem]{Remark}
\newtheorem{example}[theorem]{Example}

\def\acknowledgement{\par\addvspace{17pt}\small\rmfamily
\trivlist\if!\ackname!\item[]\else
\item[\hskip\labelsep
{\bfseries\ackname}]\fi}

\numberwithin{equation}{section}

\newcommand{\R}{\ensuremath \mathbb{R}}

\newcommand{\prox}{\ensuremath{\text{\sf prox}}}

\DeclareMathOperator*{\dom}{\ensuremath{\text{\rm dom}}}

\author{Cristian Vega}
\title{\sffamily Fast iterative regularization by reusing data
\footnote{Contact author: 
Cristian Vega Cere\~no, {\ttfamily cristian.vega.14@sansano.usm.cl},
}}
\author{Cristian Vega$^1$ 
\\[5mm]
\small MaLGa, DIMA, Universit\'a degli Studi di Genova \\
}

\date{\ttfamily \today}
\newcommand{\scal}[2]{{\left\langle{{#1}\mid{#2}}\right\rangle}}

\newcommand{\SSS}{\ensuremath{{\mathcal S}}}

\newcommand{\spn}{\ensuremath{\operatorname{span}}}

\newcommand{\Id}{\ensuremath{\operatorname{Id}}\,}
\newcommand{\RR}{\ensuremath{\mathbb{R}}}

\newcommand{\ran}{\ensuremath{\operatorname{ran}}}

\newcommand{\NN}{\ensuremath{\mathbb N}}

\newcommand{\Fix}{\ensuremath{\operatorname{Fix}}}

\newcommand{\ri}{\ensuremath{\operatorname{ri}}}

\newcommand{\vertiii}[1]{{\left\vert\kern-0.25ex\left\vert\kern-0.25ex\left\vert #1 
\right\vert\kern-0.25ex\right\vert\kern-0.25ex\right\vert}}
\newcommand{\normi}{\vert\kern-0.25ex\vert\kern-0.25ex\vert}

\ProvideTextCommand{\DJ}{OT1}{\raisebox{0.25ex}{-}\kern-0.4em D}

\numberwithin{equation}{section}
\usepackage{graphicx}
\definecolor{Mycolor1}{RGB}{193,124,250}
\newcommand{\dal}{u} 
\newcommand{\prim}{x} 
\newcommand{\proj}{p}
\newcommand{\prop}{v} 
\newcommand{\Dal}{U} 
\newcommand{\Prim}{X} 
 
\newcommand{\pSigma}{\Sigma}
\newcommand{\gamedio}{\Gamma^{\frac{1}{2}}}
\newcommand{\gamediomenos}{\Gamma^{-\frac{1}{2}}}
\newcommand{\gammmadot}{\Gamma}

\usepackage[backend=biber, style=ieee, citestyle=numeric-comp, sorting=nty]{biblatex}
\begin{document}

\title{ {\sffamily Fast iterative regularization by reusing data}}
\author{Cristian Vega\thanks{MaLGa, DIMA, Universit\'a degli Studi di Genova
        ({\tt cristian.vega@edu.unige.it}).}\ \ Cesare Molinari  \thanks{Istituto Italiano di Tecnologia, MaLGa, DIMA, Universit\'a degli Studi di Genova ({\tt molinari@dima.unige.it })} \\ 
           Lorenzo Rosasco  \thanks{Istituto Italiano di Tecnologia, MaLGa, DIBRIS, Universit\`a degli Studi di Genova, Center for Brains, Minds and Machines, Massachussets Institute of Technology ({\tt lorenzo.rosasco@unige.it })}\\  and Silvia Villa\thanks{MaLGa, DIMA, Universit\'a degli Studi di Genova ({\tt silvia.villa@unige.it}).}}
\date{\ttfamily \today}

\maketitle

\begin{abstract}
\hspace{-5mm}Discrete inverse problems correspond to solving a system of equations in a stable way with respect to noise in the data.  A typical approach to enforce uniqueness and select a meaningful solution is  to introduce a  regularizer. While for most applications the regularizer is  convex, in many cases it is not  smooth nor strongly convex. In this paper,  we propose and study two new iterative regularization methods, based on a primal-dual algorithm, to solve inverse problems efficiently.  Our analysis, in the noise free case, provides convergence rates for the Lagrangian and the feasibility gap. In the noisy case, it provides stability bounds and early-stopping rules with theoretical guarantees. The main novelty of our work is the exploitation of some a priori knowledge  about the solution set, i.e. redundant information. More precisely we show that the linear systems can be used more than once along the iteration. Despite the simplicity of the idea, we show that this procedure brings surprising advantages in the numerical applications. 
We discuss various approaches to take advantage of redundant information, that are at the same time consistent with our assumptions and flexible in the implementation. Finally, we illustrate our theoretical findings with numerical simulations for robust sparse recovery and  image reconstruction through total variation. We confirm the efficiency of the proposed procedures, comparing the results with state-of-the-art methods.

\end{abstract}

\vspace{1ex}

\section{Introduction}
Many applied problems require the estimation of a quantity of interest from noisy linear measurements, for  instance compressed sensing \cite{candes2006robust,candes2006near,donoho2006compressed,rudelson2005geometric,tsaig2006extensions}, image processing \cite{rudin1992nonlinear,osher1990feature,rudin1994total,chambolle2004algorithm,chambolle1997image,osher2005iterative,xiao2010dual}, matrix completion \cite{cai2010singular,candes2010matrix,candes2009exact,molinari2021iterative}, and various problems in machine learning \cite{shalev2014understanding,moulines2011non,rosasco2014learning,duchi2009efficient,bauer2007regularization,xiao2010dual,yao2007early}.
In all these problems, we are interested in finding stable solutions to an equation where the accessible data are corrupted by noise. This is classically achieved by regularization \cite{engl1996regularization}. The most classical procedure in the literature is Tikhonov (or variational) regularization \cite{engl1996regularization}, and consists in minimizing the sum of an error term on the residual of the equation plus a regularizer, which is explicitly added to the objective function. The regularizer entails some a priori knowledge or some desired property of the solutions that we want to select. A trade-off parameter is then introduced to balance the fidelity and the regularizer. In practice, this implies that the optimization problem has to be solved many times for different values of the parameter. Finally, a parameter - and the correspondent solution - is chosen accordingly to the performance with respect to some criteria, such as Morozov discrepancy principle \cite{engl1996regularization} or, popular technique in machine learning, cross-validation on left-out data \cite{steinwart2008support,golub1979generalized}. \\

An efficient alternative to explicit regularization is offered by iterative regularization, also known as implicit regularization \cite{engl1996regularization,burger2007error,boct2012iterative,bachmayr2009iterative}. The chosen regularizer is minimized under the constraint given by the equation, but with the available data affected by noise. A numerical algorithm to solve the optimization problem is chosen and early stopped, to avoid convergence to the noisy solution. Running the iterative procedure until convergence would give an undesired noisy solution. In this setting, the number of iterations plays the role of the regularization parameter. The best performing iterate, according to some a priori criterion (for instance, cross-validation), is then considered as the regularized solution. This procedure is very efficient when compared to explicit regularization, because it requires to solve only one optimization problem and not even until convergence.
\ \\

In this paper we are interested in iterative regularization procedures via early stopping. First we focus on linearly constrained minimization problems, when the regularizer is only convex, but not necessarily smooth nor strongly convex. The main novelty of this work is the design and analysis of two new iterative regularization methods based on primal-dual algorithms \cite{Chambolle_Pock11,Condat13,Vu13}, which perform one minimization step on the primal variable followed by one on the dual, to jointly solve the primal and the dual minimization problems. Primal-dual algorithms are computationally efficient, as only matrix-vector multiplications and the calculation of a proximity operator are required. In order to design our algorithms, we adapt the framework presented in \cite{briceno2021random} to the context of inverse problems. The key idea is to reuse data constraint at every iteration of the primal-dual algorithm, by activations of the redundant information available. The first method that we propose is a primal-dual algorithm (\ref{A: PDSP}) with additional activactions of the linear equations. We propose different variants of this procedure, depending on the extra activation step. For instance, we are able to exploit the data constraints more than once at every iteration via gradient descent, with fixed or adaptive step size. The second method is a dual-primal algorithm (\ref{A: DPSP}) where a subset containing the dual solutions is activated at each step. This subset is not affected by the noise in the data and is usually determined by a finite number of independent constraints. 

This formulation may seem artificial or inefficient. However, while maintaining an easy implementation, our methods achieve better numerical performances and considerable speed-ups with respect to the vanilla primal-dual algorithm. 
We extend to the noisy case the techniques studied in \cite{briceno2019projected,briceno2021random} for the exact case. The assumptions on the noise are the classical ones in inverse problems, see e.g. \cite{matet2017don,calatroni2021accelerated,burger2007error,molinari2021iterative}.
We generalize the results in  \cite{molinari2021iterative}, by including in the primal-dual procedure a diagonal preconditioning and an extra activation step. 
Since we are in a non-vanishing  noisy regime, it is not reasonable to expect the convergence of the iterations to the solution set of the noise free problem, thus we provide an early stopping criterion to recover a stable approximation of an ideal solution, in the same spirit of \cite{matet2017don,calatroni2021accelerated,burger2007error,raskutti2014early,zhang2005boosting,yao2007early,blanchard2010optimal,bartlett2007adaboost}.
The early stopping rule is derived from  theoretical  stability bounds and feasibility gap rates for both algorithms, obtaining implicit regularization  properties similar to those stated in \cite{molinari2021iterative} and \cite{matet2017don}. Theoretical results are complemented by numerical experiments for robust sparse recovery and total variation,  showing that state-of-the-art performances can be achieved with considerable computational speed-ups.\\

\textbf{Related works.} In this section, we briefly discuss the literature about variational  and iterative  regularization techniques.
Tikhonov regularization has been introduced in \cite{tihonov1963solution}. See  also \cite{engl1996regularization,benning2018modern} and references therein for an extensive treatment of the topic.  
The most famous iterative regularization method is the Landweber algorithm \cite{landweber1951iteration,engl1996regularization}, namely gradient descent on the least squares problem. Duality theory in optimization gives another interpretation which sheds light on the regularizing properties of this procedure. Indeed, consider the problem of minimizing the squared norm under the linear constraint. Running gradient descent on its dual problem and mapping back to the primal variable, we obtain exactly the Landweber method. This provides another explanation of why the iterates of Landweber algorithm converge to the minimal norm solution of the linear equation. Stochastic gradient descent on the previous problem is the generalization of the Kaczmarz method \cite{lorenz2008convergence,schlor19}, which consists in applying cyclic or random projections onto single equations of the linear system.   Accelerated and diagonal versions are also discussed in \cite{engl1996regularization,neubauer2017nesterov} and \cite{bakushinsky2005iterative,kaltenbacher2008iterative,scherzer1998modified}, respectively. 
The regularization properties of other optimization algorithms for more general regularizers have been also studied. If strong convexity  is assumed, mirror descent \cite{beck2003mirror,nemirovskij1983problem} can also be interpreted as gradient descent on the dual problem, and its regularization properties (and those of its accelerated variant) have been studied in  \cite{matet2017don}. Diagonal approaches \cite{bahraoui1994convergence} with a regularization parameter that vanishes along the iterations  have been studied in \cite{garrigos2018iterative}, see  \cite{calatroni2021accelerated} for an accelerated version.  Another common approach relies on the linearized Bregman iteration \cite{yin2008bregman,yin2010analysis, xiao2010dual, osher2005iterative}, which has found applications in compressed sensing \cite{cai2009linearized,osher2010fast,yin2008bregman}  and image deblurring \cite{cai2009linearized}. However, this method requires to solve non-trivial minimization problems at each iteration. For convex, but not strongly convex regularizers, the regularization properties of a primal-dual algorithms have been investigated in \cite{molinari2021iterative}.
\ \\

The rest of the paper is organized as follows. In  Section~\ref{sec:NB} we introduce the notation jointly with its mathematical background. In Section~\ref{s: MPA} we present the main problem and propose five classes of algorithms to solve it numerically. In Section~\ref{s: MR} we derive stability and feasibility gap bounds and related early stopping rules. In Section~\ref{s: app}  we verify the performance of the algorithm on two numerical applications: robust sparse recovery problem and image reconstruction by total variation. Finally, we provide some conclusions.
\section{Notation and background}
\label{sec:NB}
First we recall some well known concepts and properties used in the paper.
\ \\

Let $X$, $Y$ be two finite-dimensional real vector spaces equipped with an inner product $\scal{\cdot}{\cdot}$ and the induced norm $\|\cdot\|^2$. We denote the set of convex, lower semicontinuous, and proper functions on $X$ by $\Gamma_{0}(X)$. The subdifferential of $F\in \Gamma_{0}(X)$ is the set-valued operator defined by
\begin{align}
\partial F\colon \  X\to 2^{X}, \quad
x\mapsto\{u\in X\hspace{0.2cm}|\hspace{0.2cm}(\forall 
y\in X)\hspace{0.2cm} F(x)+\langle y-x\mid\hspace{0mm} 
u\rangle\leq F(y)\}. \label{d: subdifferential}
\end{align}

If the function $F$ is G\^ateaux differentiable at the point $x$, then $\partial F(x)=\{\nabla F(x)\}$\hspace{2mm}\cite[Proposition 17.31 (i)]{bauschke2011convex}. In general, for $F\in 
\Gamma_{0}(X)$, it holds that $(\partial F)^{-1}=\partial F^{*}$ \hspace{2mm}\cite[Corollary 16.30]{bauschke2011convex}, where $F^{*}\in \Gamma_{0}(X)$ is the conjugate function of $F$, defined by $F^{*}(x):=\sup _{u \in 
X} \ \scal{x}{u}- F(u)$. 
\ \\

For every self-adjoint positive definite matrix $\pSigma$, we define the proximity operator of $F$ relative to the metric induced by $\|\cdot\|_{\pSigma}^2:=\scal{\cdot}{\pSigma \cdot}$ as $\operatorname{prox}^{\pSigma}_{F}=(\Id+\pSigma\partial F)^{-1}$. If $\pSigma=\sigma\Id$ for some real number $\sigma>0$, it is customary to write $\prox_{\sigma F}$ rather than $\operatorname{prox}^{\pSigma}_{F}$ . The projector operator onto a nonempty closed convex set $C \subseteq X$ is denoted by $P_{C}$. If we define the indicator $\iota_{C}\in\Gamma_{0}(X)$ as the function that is $0$ if $x$ on $C$ and $+\infty$ otherwise, then $\prox_{\iota_{C}}=P_{C}$. Moreover, if $C$ is a singleton, say $C=\{b\}$, we have that $\iota^{*}_{\{b\}}(u)=\scal{u}{b}$. The relative interior of $C$ is $\ri(C)=\left\{ x\in C\mid \RR_{++} (C-x)= \spn (C-x)\right\},$ where $\RR_{++}C=\left\{\lambda y\mid (\lambda >0)\wedge(y\in 
C)\right\}$ and $\spn(C)$ is the smallest linear subspace of $X$ containing $C$. 
\ \\

Given $\alpha~\in~]0, 1[ $, an operator $T : \ X \rightarrow 
X$ is $\alpha$-averaged non-expansive iff
$$(\forall x\in X )(\forall y\in X)\hspace{3mm} \|T 
x - Ty\|^2 \leq \|x - y\|^2 -\frac{1-\alpha}{\alpha}\|(\Id-T)x- (\Id-T)y\|^2,$$ and it is quasi-non-expansive iff:
$$(\forall x\in X )(\forall y\in \Fix T)\hspace{3mm} \|T 
x - y\|^2 \leq \|x - y\|^2,$$  where the set of fixed points of $T$ is defined by $\Fix T=\{x\in X\mid Tx=x\}$. For further results on convex analysis and operator theory, the reader is referred to \cite{bauschke2011convex}.
\ \\

For a real matrix $A\in\RR^{d\times p}$, its operator norm is denoted by $\|A\|$ and its adjoint by $A^{*}$. We define the Frobenius norm of $A$ as $\|A\|^{2}_{F}:=\sum_{i=1}^{d}\|a_{i}\|^2$, where, for every $i\in[d]:=\{1,\ldots,d\}$, $a_{i}$ denotes the $i$-th row of $A$. We also denote by $A_i$ the $i$-th column of $A$. We denote by $\ran(A)$ and $\ker(A)$ the range and the kernel of $A$, respectively.

\section{Main problem and algorithm}
\label{s: MPA}
Many applied problems require to estimate a quantity of interest $x\in\RR^p$ based on linear measurements $b=Ax$, for some matrix $A\in\RR^{d \times p}$. For simplicity, we carry the analysis in this finite dimensional case, but note that it can be easily extended to the infinite dimensional setting. A standard approach to obtain the desired solution is to assume that it is a minimizer of the following linearly constrained optimization problem:
\begin{align}
\min_{x\in \RR^p}   J(x)	\hspace{4mm}
\text{s.t.}				\hspace{4mm} Ax=b,
 \tag{$\mathcal{P}$}
\label{P: problem}
\end{align}
where $J\in\Gamma_{0}(\RR^p)$ encodes  a priori information on the solution and is usually hand-crafted. Typical choices are:  the squared norm \cite{engl1996regularization}; the elastic net regularization \cite{matet2017don}; the $\ell^{1}$-norm \cite{candes2006robust,candes2006near,donoho2006compressed,tsaig2006extensions}; the total variation \cite{rudin1992nonlinear,osher1990feature,rudin1994total,chambolle2004algorithm}. Note that, in the previous examples, the first two regularizers are strongly convex, while the second two are just convex and non-smooth.
 \\ \\ If we use the indicator function of  $\{b\}$, \eqref{P: problem} can be written equivalently as
 \begin{align}
        \min_{x\in \RR^p}  J(x)+\iota_{\{b\}}(Ax).
        \label{P: Pc}
    \end{align}
 
 We denote by $\mu$ the optimal value of $\eqref{P: problem}$ and by $\SSS$ the set of its minimizers. We assume that $\SSS\neq\emptyset$. In order to build our regularization procedure, we consider the Lagrangian functional for problem $\eqref{P: problem}$:
    \begin{equation}
     \label{e:saddle point}
     \mathcal{L}(\prim,\dal):=J(\prim)+\scal{\dal}{A\prim-b}.
     \end{equation}
This approach allow us to split the contribution of the non-smooth term $J$ and the one of the linear operator $A$, without requiring to compute the projection on the set $C:=\{x\in\RR^{p}\mid Ax=b\}$. We define the set of saddle points of $\mathcal{L}$ as
     \begin{equation}
         \mathcal{Z}=\left\{(\prim, \dal)\in \RR^p\times\RR^d: \ \mathcal{L}(\prim,v)\leq \mathcal{L}(\prim,\dal)\leq \mathcal{L}(y,\dal) \ \ \forall(y,v)\in \RR^p\times\RR^d \right\}. 
   \end{equation}
The set $\mathcal{Z}$ is characterized by the first-order optimality condition:
     \begin{align}
      \mathcal{Z}= \left\{(x,u)\in \RR^p\times\RR^d:   0\in\partial J(x)+A^{*}u\hspace{2mm}\text{ and }\hspace{2mm}Ax=b\right\}.
     \end{align}
 In the following, we always assume that $\mathcal{Z}\neq \emptyset.$ 
 \ \\
 
\begin{remark}[Saddle points and primal-dual solutions] The set of saddle points is ensured to be nonempty 
when some qualification condition holds (see \cite[Proposition 6.19]{bauschke2011convex} special cases), for instance when 
\begin{align}
b\in \ri\left(A\left(\dom J\right)\right).
\label{e: qualication conditions}
\end{align}
Observe that the objective function of \eqref{P: problem} is the sum of two functions in $\Gamma_{0}(\RR^p)$  where one of the two is composed with a linear operator.  This formulation is suitable to apply  Fenchel-Rockafellar duality. Recalling that $\iota^{*}_{\{b\}}(u)=\scal{u}{b}$ \cite[Example 13.3(i)]{bauschke2011convex}, the dual problem of \eqref{P: problem} is given by 
    \begin{align}
        \min_{u\in \RR^d}  J^{*}(-A^*u)+\scal{u}{b}.
        \label{P: Pd}
        \tag{$\mathcal{D}$}
    \end{align}
We denote its optimal value  by $\mu_{*}$ and by $\SSS^{*}$ its set of minimizers. Then, $\mathcal{Z}\subseteq\SSS\times \SSS^{*}$, and equality holds if  \eqref{e: qualication conditions} is satisfied \cite[Proposition 19.21 (v)]{bauschke2011convex}.\ \\

In addition, condition \eqref{e: qualication conditions}   implies that problem \eqref{P: Pd} has a solution. Then under the qualification condition, since we assumed that  $S\neq\emptyset$, we derive also that $\mathcal{Z}\neq \emptyset$. 
\end{remark}
 \ \\ 
 
In practical situations, the exact data $b$ is unknown and only a noisy version is accessible. 
Given a noise level $\delta\geq0$, we consider a worst case scenario, where the error is deterministic and the accessible data $b^\delta$ is such that \begin{equation}
    \|b^{\delta}-b\|\leq\delta.
\end{equation} This is the classical model in inverse problems \cite{engl1996regularization,kaltenbacher2008iterative}. The solution set of the inexact linear system $Ax=b^{\delta}$ is denoted by $C_{\delta}$. Analogously,  we denote by $\SSS_\delta$ and $\SSS_\delta^*$ the set of primal and dual solutions with noisy data.  It is worth pointing out that, if $b^\delta\not\in\ran(A)$, then $\SSS_\delta\subseteq C_{\delta}= \emptyset$ but our analysis  and bounds still hold. 

\subsection{Primal-Dual Splittings with a priori Information}\label{s:pd}
In this section, we propose an iterative regularization procedure to solve problem \eqref{P: problem}, based on a primal-dual algorithm with preconditioning and arbitrary activations of a predefined set of operators. 
While the use of primal-dual algorithms \cite{chambolle2011first} as iterative regularization methods is somewhat established \cite{molinari2021iterative}, in this paper we focus on the possibility of reusing the data constraints during the iterations. This idea was originally introduced in \cite{briceno2021random}, where the authors studied the case in which the exact data is available, and consists in the activation of extra operators, that encode information about the solution set, to improve the feasibility of the updates. In our setting, we can reuse data constraints, and we project, in series or in parallel, onto some equations given by the (noisy) linear  constraint. But we will show that other interesting choices are possible, as projections onto the set of dual constraints. 
\ \\

More formally, for $i\in [m]$, we consider a finite number of operators $T_i\colon \ \RR^p\to \RR^p$ or $T_i\colon \ \RR^d\to \RR^d$, such that the set of noisy primal solutions is contained in $\Fix T_i$ for every $i\in [m]$. We refer to this as a redundant a priori information. A list of  operators suitable to our setting (and with practical implementation) can be found in Section~\ref{s: app}.
\ \\

The primal-dual algorithms with reuse of constraints which are given in  Table~\ref{t:algos} are a preconditioned and deterministic version of the one proposed in \cite{briceno2021random} applied to the case of linearly constrained minimization.
\begin{table}[ht!]
\label{t:algos}
    \centering \resizebox{0.8\columnwidth}{!}{
    \begin{tabular}{c c}
\hspace{-5mm} \begin{tabular}{|m{65mm}|}
          \hline   Primal-Dual splitting with activations  \\  \hline\vspace{2mm} \textbf{Input}:  $(\bar{\proj}^0,\proj^0,\dal^0)\in\RR^{2p}\times\RR^{d}$.\\\vspace{-0mm} \begin{flushleft}\textbf{For} $k=1,\ldots,\text{N:}$\end{flushleft}\vspace{-4mm}\\ \vspace{-10mm}
        \begin{align}
    \begin{array}{l}
\dal^{k+1}= \dal^k+\Gamma( A\bar{\proj}^k-b^\delta)\\
\prim^{k+1}=\prox^{\pSigma }_{J}(\proj^k-\pSigma A^*\dal^{k+1})\\
\proj^{k+1}=T_{\epsilon_{k+1}}\prim^{k+1}\\
\bar{\proj}^{k+1}=
\proj^{k+1}+ \prim^{k+1}-\proj^{k},
\end{array}  \label{A: PDSP}\tag{PDA}\end{align}\vspace{-4mm}\\ \vspace{0mm}\begin{flushleft}\textbf{End}\end{flushleft}  \\  \hline
      \end{tabular} 
     & \hspace{-4mm} \begin{tabular}{|m{65mm}|}
          \hline   Dual-Primal splitting with activations  \\  \hline\vspace{2mm} \textbf{Input}: $(\prim^{0},\bar{\prop}^{0},\dal^0)\in \RR^{p}\times\RR^{2d}$.\\\vspace{-0mm} \begin{flushleft}\textbf{For} $k=1,\ldots,\text{N:}$\end{flushleft}\vspace{-4mm}\\ \vspace{-10mm}
        \begin{align}
    \begin{array}{l}
\prim^{k+1}=\prox^{\pSigma }_{J}(\prim^k-\pSigma A^*\Bar{\prop}^{k})\\
\dal^{k+1}= \prop^k+\Gamma( A\prim^{k+1}-b^\delta)\\
\prop^{k+1}=T_{\epsilon_{k+1}}\dal^{k+1}\\
\bar{\prop}^{k+1}=
\prop^{k+1}+ \dal^{k+1}-\prop^{k},
\end{array} \label{A: DPSP}\tag{DPA}\end{align}\vspace{-4mm}\\ \vspace{0mm}\begin{flushleft}\textbf{End}\end{flushleft}  \\  \hline
      \end{tabular}
 \\
    \end{tabular}}\caption{Proposed algorithms for iterative regularization.} \label{T:Alg}
    \end{table}
We first focus on the Primal-Dual splitting. It is composed by four different steps, to be performed in series. The first step is the update of the dual variable, in which the residuals to the linear equation $Ax=b^\delta$ are accumulated after preconditioning by the operator $\Gamma$. The second step is an implicit prox-step, with function $J$ and norm $\|\cdot\|_{\pSigma^{-1}}$, on the primal variable. The third one is the activation of the operator related to reusing data constraint, on the primal variable. Finally, the last step is an extrapolation again on the primal variable. Notice that, if no operator is activated, it corresponds simply to $\bar{\proj}^{k+1}=
2 \prim^{k+1}-\prim^{k}$, that is the classical update in primal-dual algorithm. On the other hand, the Dual-Primal Splitting algorithm, except for permutation in the order of the steps, differs from the previous one because the activation of the operator is done not on the primal variable but on the dual one. Indeed, Lemma \ref{L: PD=DP} establishes that, without the activation of the operator, there is an equivalence between the primal variables generated by \ref{A: PDSP} and the ones generated by \ref{A: DPSP}. \\

\begin{remark}
As already mentioned, our analysis can be easily extended to infinite dimensional problems. In particular, note that the primal-dual algorithms above can be formulated exactly in the same way for infinite dimensional problems. The convergence guarantees of the plain methods in Hilbert and Banach spaces have been studied in \cite{Condat13,Vu13,silveti2021stochastic}. 

Another possible extension of the algorithm, that we do not analyse explicitly in this work, is related with the stochastic version of primal-dual; see \cite{chambolle2018stochastic,alacaoglu2019convergence,gutierrez2021convergence}. On the other hand, note that in \eqref{A: PDSP} the redundant activation of the data constraint is arbitrary. In particular, it can be chosen in a stochastic way at every iteration.

\end{remark}

\ \\
In the following, we list the assumptions that we require on the parameters and the operators involved in the algorithm. 

\begin{assumption}\label{A: structured error1} 
Consider the setting of \ref{A: PDSP} or \ref{A: DPSP}:
\begin{enumerate}
\item[($A1$)]\label{A: structured error1a}  The preconditioners $\pSigma\in\RR^{p\times p}$ and $\Gamma\in\RR^{d\times d}$ are two diagonal positive definite matrices such that
\begin{align}
    0<\alpha:=1-\|\gamedio A\pSigma^{\frac{1}{2}}\|^2.
\label{c: ConditionL D1}
 \end{align}
\item[($A2$)]\label{A: structured error1b} For every $k\in\mathbb{N}$, $\epsilon_k\in[m]$.
\end{enumerate}
Consider the setting of \ref{A: PDSP}: 
\begin{enumerate}
\item[($A3$)]\label{A: structured error2} $\left\{T_{i}\right\}_{i\in [m]}$ is a family of operators from $\RR^{p}$ to $\RR^{p}$ and  for every $i\in [m]$:
\begin{enumerate}
\item $\Fix T_{i}\supseteq\SSS_{\delta} \supseteq\emptyset $;
\item there exist  $e_i\geq 0$ such that, for every $\prim\in\RR^p$ and $\bar{\prim}\in \SSS$,
\begin{align}
\label{A: pitagoras error 1}
    \|T_i\prim-\bar{\prim}\|_{\pSigma^{-1}}^2\leq \|\prim-\bar{\prim}\|_{\pSigma^{-1}}^2+e_i\delta^{2}.\end{align} We denote by $e=\max_{i\in[m]} e_i$.
\end{enumerate} \label{c: ConditionL D3}
\end{enumerate}

Now consider the setting of \ref{A: DPSP}:
\begin{enumerate}
\item[($A4$)\label{A: structured error3}] $\left\{T_{i}\right\}_{i\in [m]}$ is a family of operators from $\RR^{d}$ to $\RR^{d}$ and  for every $i\in [m]$:
\begin{enumerate}
\item   $\Fix T_{i}\supseteq\SSS^{*}_{\delta}\supseteq \Fix T_{i}\emptyset$;
\item   for every $u\in\RR^d$ and $\bar{u}\in \SSS^{*}_{\delta}$,
\begin{align}
\label{A: pitagoras error 2}
    \|T_iu-\bar{u}\|_{\Gamma^{-1}}^2\leq \|u-\bar{u}\|_{\Gamma^{-1}}^2.\end{align} 
\end{enumerate} \label{c: ConditionL D4}
\end{enumerate}
\end{assumption}
\begin{remark}[Hypothesis about the operators]
If Assumptions A3-(a) holds and $\delta=0$, Assumptions A3-(b) is implied by quasi-nonexpansivity of $T_i$ on $\SSS$. The previous is a weaker condition than the one proposed in \cite{briceno2021random}, where, due to the generality of the setting,  $\alpha$-averaged non-expansive operators are needed. A similar reasoning applies to Assumption A4.
\end{remark} 

\section{Main results}
\label{s: MR}
In this section, we present and discuss the main results of the paper. We derive stability properties of primal-dual and dual-primal splitting  for linearly constrained optimization with a priori information.  
\ \\

First, we define the averaged iterates and the square weighted norm induced by  $\pSigma$ and $\Gamma$ on $\RR^p\times\RR^d$, namely
\begin{align}
    \left(\Prim^n,\Dal^n\right):=\frac{\sum_{k=1}^{n}z^{k}}{n}\hspace{1mm} \text{ and }\hspace{1mm}  V(z):=\frac{\|\prim\|_{\pSigma^{-1}}^2}{2}+\frac{\|\dal\|_{\Gamma^{-1}}^2}{2},
    \label{D: Wnorm}
\end{align} where $z^{k}:=(\prim^{k},\dal^{k})$ is the $k$-th iterate and $z:=(\prim,\dal)$ is a primal-dual variable. We also recall the the definition of the Lagrangian as $\mathcal{L}(\prim,\dal):=J(\prim)+\scal{\dal}{A\prim-b}$

The first result establishes the stability properties of algorithm \ref{A: PDSP}, both in terms of Lagrangian and feasibility gap. We recall that here we use activation operators based on the noisy feasibility constraints in the primal space, namely the set $C_\delta$.
 \begin{theorem}\textbf{}
    \label{Th:PPD}Consider the setting of \ref{A: PDSP} under Assumptions A1, A2, and A3.  Let $(\bar{\proj}^0,\proj^{0},\prim^{0})\in\RR^{2p}\times\RR^{d}$ be such that  $\proj^0=\bar{\proj}^{0}$.  Then, for every  $z~=~(\prim,\dal)~\in~\mathcal{Z}$ and for every $N\in\NN$, we have 
     \begin{align}
        \mathcal{L}\left(\Prim^{N},\dal\right)- \mathcal{L}\left(\prim,\Dal^{N}\right)\leq& \frac{V(z^{0}-z)}{N}+\frac{2N\|\Gamma^{\frac{1}{2}}\|^2\delta^{2}}{\alpha}+\delta\|\Gamma^{\frac{1}{2}}\|\left(\frac{2 V(z^{0}-z)}{\alpha}\right)^{\frac{1}{2}}\nonumber\\ &+\delta\|\Gamma^{\frac{1}{2}}\|\left(\frac{ N e\delta^2}{\alpha}\right)^{\frac{1}{2}}+\frac{e\delta^2}{2} \hspace{2mm}\label{e: DG}
    \end{align}    and
    \begin{align}
            \|A\Prim^N-b\|^2\leq&\frac{16N\|\Gamma\|\|\Gamma^{-1}\|\delta^{2}}{\alpha^{2}}+8\delta\|\Gamma^{-1}\|\left(\frac{2\|\Gamma\| V(z^{0}-z)}{\alpha^3}\right)^{\frac{1}{2}}+8\delta^{2}\|\Gamma^{-1}\|\left(\frac{ \|\Gamma\|e N}{\alpha^3}\right)^{\frac{1}{2}\nonumber}\\
         &+\frac{8\|\Gamma^{-1}\|V(z^{0}-z)}{N\alpha}+2\delta^{2}+\frac{4\|\Gamma^{-1}\|e\delta^2}{\alpha}, \label{e: RN}
    \end{align}
    where we recall that the constants $\alpha$ and $e$ are defined in Assumptions A1 and A3, respectively.
    \end{theorem}
The proof of Theorem~\ref{Th:PPD} is given in the Appendix, Section \ref{Proof:PPD}. The proof  combines and extends the techniques  developed in \cite{briceno2021random} and \cite{molinari2021iterative}, based on the firm non-expansivity of the proximal point operator and discrete Bihari's lemma to deal with the error; see also \cite{rasch2020inexact}.

In the next result,  we establish upper bounds for the Lagrangian and feasibility gap   analogous to those proposed in Theorem \ref{Th:PDP}, but for algorithm PDA. The main difference is that now the activation step is based on a priori information in the dual space $\RR^d$, and not on $C_{\delta}$. This set is represented by the intersection of fixed point sets of a finite number of operators and encodes some knowledge about the dual solution.
\begin{theorem}
\label{Th:PDP}
Consider the setting of \ref{A: PDSP} under Assumptions A1, A2, and A4. Let $(\bar{\proj}^0,\prop^{0},\prim^{0})\in \RR^{2d}\times\RR^{p}$ be such that 
$\prop^0=\bar{\proj}^{0}$. Then, for every  $z~=~(\prim,\dal)~\in~\mathcal{Z}$ and for every $N\in \NN$, we have that\\
 \scalebox{0.99}{\parbox{\linewidth}{\begin{align}
 \label{B: Dual-lagrangian}
    \mathcal{L}\left(\Prim^{N},\dal\right)- \mathcal{L}\left(\prim,\Dal^{N}\right)\leq& \frac{V(z^{0}-z)}{N}+2\|\Gamma^{\frac{1}{2}}\|^{2} N\delta^{2}+\|\Gamma^{\frac{1}{2}}\|\delta\left(2V(z^{0}-z)\right)^{\frac{1}{2}}, \hspace{2mm}\end{align}}}\\ 
 and  \\\scalebox{0.99}{\parbox{\linewidth}{\begin{align}
\label{B: Dual-feasibility}
        \|A\Prim^N-b\|^2\leq& \frac{8\|\Gamma^{\frac{1}{2}}\|^{2}\|\Gamma^{-1}\| N\delta^{2}}{\alpha}+\frac{4\|\Gamma^{\frac{1}{2}}\|\|\Gamma^{-1}\|\delta\left(2V(z^{0}-z)\right)^{\frac{1}{2}}}{\alpha}\nonumber\\&+\frac{4\|\Gamma^{-1}\|V(z^{0}-z)}{N\alpha}+2 \delta^{2}. 
    \end{align}}}
    \\ where we recall that the constants $\alpha$ is defined in Assumptions A1.
\end{theorem}
The proof is given in the Appendix, Section \ref{Proof:PDP}.\\
\\ \

First, we comment the chosen optimality measures. If the penalty is strongly convex, the Bregman divergence is an upper bound of the squared norm of the difference between the reconstructed and the ideal solution, while if $J$ is only convex, the Bregman divergence gives only limited information. As discussed in \cite{rasch2020inexact}, the Lagrangian gap is equivalent to the Bregman distance of the iterates to the solution, and in general it is a very weak convergence measure. For instance, in the exact case, a vanishing Lagrangian gap does not imply that cluster points of the generated sequence are primal solutions. However, as can be derived from \cite{molinari2021iterative}, a vanishing Lagrangian gap coupled with vanishing feasibility gap implies that every cluster point of the primal sequence is a solution of the primal problem. 

In both theorems, the established result ensures that the two optimality measures can be upper bounded with the sum of two terms. The first one, which can be interpreted as an optimization error, is of the order $\mathcal{O}(N^{-1})$ and so it goes to zero as $N$ tends to $+\infty$. Note that, in the exact case $\delta=0$, only this term is present and both the Lagrangian and the feasibility gap are indeed vanishing, guaranteeing that every cluster point of the sequence is a primal solution. The second term, which can be interpreted as a stability control, collects all the errors due to the perturbation of the exact datum and takes also into account the presence of the activation operators $T$, when the reuse data constraint is noisy. It is an increasing function of the number of iterations and the noise level $\delta$.

\begin{remark} 
Theorems~\ref{Th:PPD} and \ref{Th:PDP} are an extension of \cite{briceno2021random}, where the authors prove that the sequence generated by the algorithms converges to an element in $\mathcal{Z}$ when $\delta=0$, but no convergence rates neither stability bounds were given. In this work, we filled the gap for linearly constrained convex optimization problems.
\ \\

Moreover, in the noise free case,  our assumptions on the additional operators $T$ are weaker than those proposed in \cite{briceno2021random}, where $\alpha$-averagedness is required. 
 For the noisy case, without the activation operators (so with $e=0$), our bounds are of the same order as \cite{molinari2021iterative} in the number of iterations and noise level.
\end{remark}

 As mentioned above, in \eqref{e: DG} and \eqref{e: RN}, when $\delta>0$ and $N\rightarrow +\infty$ the upper bounds for the \ref{A: PDSP} iterates tend to infinity and the iteration may not converge to the desired solution.  The same comment can be made for the \ref{A: DPSP} iterates, based on \eqref{B: Dual-lagrangian} and \eqref{B: Dual-feasibility}.
 In both cases, to obtain a minimal reconstruction error, we need to impose a trade off between convergence and stability.
 The next corollary introduces an early stopping criterion, depending only on the noise level and leading to stable reconstruction. 
 
\begin{corollary}\label{ESPDA} (Early-stopping). Under the assumptions of Theorem \ref{Th:PPD} or Theorem~\ref{Th:PDP}, choose $N={c}/{\delta}$ for some $c>0$. Then, for every  $z~=~(\prim,\dal)~\in~\mathcal{Z}$, there exist constants $C_1$, $C_2$, and $C_3$ such that
 \begin{align}
    \mathcal{L}\left(\Prim^{N},\dal\right)- \mathcal{L}\left(\prim,\Dal^{N}\right)\leq&   C_1\delta\nonumber\\
        \|A\Prim^N-b\|^2\leq& C_2\delta+ C_3\delta^{2}.
        \label{e: earlystoppingp}
    \end{align}
\end{corollary}

The early stopping rule prescribed above is computationally efficient, in the sense that the number of iterations is proportional to the inverse of the noise level. In particular, if the error $\delta$ is small then more iterations are useful, while if $\delta$ is big, it is convenient to stop sooner. So, the number of iterations plays the role of a regularization parameter. 
Using the early stopping strategy proposed above, we can see that the error in the data transfers to the error in the solution with the same noise level, which is the best that one can expect for a general operator $A$.

\begin{remark}\textbf{Comparison with Tikhonov regularization.} The reconstruction properties of our proposed algorithm are comparable to the ones obtained using Tikhonov regularization \cite{engl1996regularization}, with the same dependence on the noise level \cite{benning2011error}. We underline that in the previous paper only the Bregman divergence is considered, and not the feasibility.

One main difference between Tikhonov and iterative regularization techniques is the fact that the Tikhonov parameter $\lambda$ is a continuous regularization parameter, while the iteration counter is a discrete one. This may be seen as a disadvantage, but usually in the practise it may be fixed with the choice of a smaller step-size in the algorithm.
On the other hand, iterative regularization is way more efficient from the computational point of view, as it requires the solution of only one optimization problem, while explicit regularization amounts to solve a family of problems indexed by the regularization parameter. Let us also note that, when $\delta$ is unknown, any principle used to determine a suitable $\lambda$ can be used to determine the stopping time.
\end{remark}

\section{Implementation details}
\label{s: app}

In this section we discuss some possible standard choices to construct non-expansive operators $T$ that satisfy our assumptions and encode some redundant information on the solution set. We first present examples for \ref{A: PDSP}, and later for \ref{A: DPSP}. 

To define the operators, we first recall the projection on a row. For every $j\in  [d]$ we denote by $a_j$ the $j$-th row of $A$ and by $P_j$ the projection onto the $j$-th linear equation; namely,
\begin{align}
   P_{j}\colon\mathbb{R}^p \mapsto \mathbb{R}^p, \hspace{2mm} \prim \mapsto    \prim+\frac{b_{j}-\scal{a_{j}}{\prim}}{\|a_{j}\|^2}a_{j}^{*}.
    \label{d: averaged}
\end{align} 
Analogously, for every $j\in [d]$, we denote by $P^\delta_{j}$ the projection operator as in the previous definition but with the noisy data $b^\delta$ instead of $b$.

We proceed to define the four families of operators proposed in this paper for \ref{A: PDSP}.
\begin{definition}\label{O: Operators} The operator $T\colon \RR^{p}\mapsto\RR^{p}$ is a 
\begin{enumerate}
    \item \textbf{Serial projection} if
    \begin{align}
  T=P^{\delta}_{\beta_{l}}\circ\cdots\circ P^{\delta}_{\beta_{1}},
    \label{d: averaged1}
\end{align} where, for every $j\in [l]$, $\beta_{j}\in [d]$.
    \item \textbf{Parallel projection} if
    \begin{align}
  T=\sum\limits_{j=1}^{l}\alpha_{j} P^{\delta}_{\beta_{j}}
    \label{d: averaged2}
\end{align} where, for every $j\in [l]$, $\beta_{j}\in [d]$ and $\left(\alpha_{j}\right)_{j=1}^{l}$ are real numbers in $[0,1]$, such that $\sum\limits_{j=1}^{l}\alpha_{j}=1$. 
 \item \textbf{Landweber operator} with parameter $\alpha$ if
    \begin{align}
   T:\mathbb{R}^p \mapsto \mathbb{R}^p, \hspace{2mm} \prim \mapsto    \prim-\alpha A^{*}(A\prim-b^\delta).
    \label{d: averaged3}
\end{align}
   where $\alpha\in ]0,\frac{2}{\|A\|^2}[$.
    \item \textbf{Landweber operator with adaptive step} and parameter $M$ if
    \begin{align}
   T\colon\mathbb{R}^p \mapsto \mathbb{R}^p, \hspace{2mm} \prim \mapsto   \left\{ \begin{array}{ll} \prim-\beta(x) A^{*}(A\prim-b^\delta)
             &  \text{\ \ \ if } A^{*}A\prim\neq A^{*}b_\delta \\
              \prim &  \text{\ \ \ otherwise.}
             \end{array}\right.
    \label{d: averaged4}
\end{align}  
where, for $M>0$, $\beta(x)=\min\left(\frac{\|A\prim-b^\delta\|^2}{\|A^{*}(A\prim-b^\delta)\|^2},M\right)$.
\end{enumerate}
\end{definition}

The next lemma states that the operators in Definition~\ref{O: Operators} satisfy Assumption A3.
\begin{lemma}
\label{L: Series Parallel}
Let $T\colon\RR^p\to\RR^p$ be one of the  operators given in Definition~\ref{O: Operators}. Then Assumption A3 holds with
 \begin{enumerate}
     \item $e_T~=\sum_{j=1}^l \frac{1}{\|a_{{\beta_j}}\|^2}$, if $T$ is a serial projection;
     \item $e_T~=\sum_{j=1}^l\frac{\alpha_{j}}{\|a_{\beta_j}\|^2}$, if $T$ is a parallel projection;
    \item   $e_T=\frac{\alpha}{2-\alpha\|A\|^2}$,  if  $T$ is the Landweber operator with parameter $\alpha$;
    \item $e_T=M$, if $T$ is the Landweber operator with adaptive step and parameter $M$.
 \end{enumerate} 
\end{lemma}
\begin{remark}\label{R: Parallel-Landweber} \textbf{Relationship between Parallel projection and Landweber operator}. A particular  parallel projection is the one corresponding to $l=d$, $\beta_{j}=j$, and $\alpha_{j}=\frac{\|a_j\|^2}{\|A\|_{F}^2}$. Then,  \eqref{d: averaged2} reduces to
\begin{equation}
    T(x)=x-\frac{1}{\|A\|_{F}^2}A^{*}(Ax-b^\delta).\label{T:Land-Paralell}
\end{equation}
Observe that, since $\|A\|\leq \|A\|_{F}$, the previous is a special case of Landweber operator with $\alpha=\frac{1}{\|A\|_{F}^2}$.
\end{remark}
\begin{remark}\textbf{Steepest descent}. Let $\bar{\prim}\in \R^p$ such that $A\bar{\prim}=b$. Then, from \eqref{d: averaged4}, we derive (see also equation \eqref{e:Steepest descentbeta} in the Appendix)
\begin{align}
    \|T\prim-\bar{\prim}\|^2&=\|\prim-\bar{\prim}\|^2-2\beta(x)\scal{b^{\delta}-b}{A\prim-b^\delta}-2\beta(x)\|A\prim-b^\delta\|^2\nonumber\\&\hspace{2mm}+\beta(x)^2\|A^{*}(A\prim-b^\delta)\|^2.\label{e:Steepest descentbeta0}
\end{align}
If $\delta=0$, then the choice of $\beta(x)$ given in \eqref{d: averaged4} minimizes  the right hand side of \eqref{e:Steepest descentbeta0}, if the minimizer is smaller than $M$. In this case, $\beta$ is chosen in order to maximize the contractivity with respect to a fixed point of $T$. While we cannot repeat the same procedure for $\delta > 0$, since we do not know $b$, we still keep the same choice. 
If $b^\delta\in \ran(A)$, then $\sup\limits_{x\in\RR^p}\frac{\|A\prim-b^\delta\|^2}{\|A^{*}(A\prim-b^\delta)\|^2}<+\infty$. 
However, in general, if $\delta > 0$, this is not true and $M$ is needed to ensure that $\beta(x)$ is bounded.  
\end{remark}
\begin{remark} 
    From a computational point of view, parallel projections and Landweber operators are more efficient than  serial projections. In particular, note that the quantity $(Ax^{k}-b^\delta)$ needs to  be computed anyway in the other steps of the algorithm.
    \end{remark}
While for the primal space the reuse data constraint that we want to exploit is clearly given by the linear constraint, for the dual is not always so. In the following we present an example related to the $\ell^1$ norm. A similar implementation can be extended to the case of $1$-homogenous penalty functions, for which the Fenchel conjugate is the indicator of a closed and convex subset of the dual space \cite[Proposition 14.11 (ii)]{bauschke2011convex}. 

\begin{example}
\label{e:dpl1}
Consider the noisy version of problem \ref{P: problem} with $J(x)= \|x\|_1$. 
Then the dual is given by
\[
\min_{u\in\RR^d} \langle b^\delta, u\rangle \,:\, |(A^*u)_i| \leq 1, \text{ for every $i\in [p]$}. 
\]
For every $i\in [p]$, set $D_i=\{u\in\RR^d\,:\, |(A^*u)_i| \leq 1\}$ and denote by $T_i$ the projection over $D_i$.
Note that this is trivial to compute, since it is the projection onto the intersection of two parallel hyperplanes. 
Clearly  Assumption A4 holds. 
Differently from the primal case, here we are projecting on exact constraints, independent from the noisy data $b^{\delta}$.
\end{example}

\section{Numerical results}

In this section, to test the efficiency of the proposed algorithms, we perform numerical experiments in two relevant settings: regularization with the $\ell^1$-norm and total variation regularization. For the $\ell^1$-norm regularization,  we compare our results with other regularization techniques. In the more complex problem of total variation we explore the properties of different variants of our procedure.

\textbf{Code statement:} All numerical examples are implemented in MATLAB\textsuperscript{\textregistered} on a laptop.
In the second experiment we also use the library Numerical tours \cite{peyre2011numerical}. The corresponding code can be downloaded at \href{https://github.com/cristianvega1995/L1-TV-Experiments-of-Fast-iterative-regularization-by-reusing-data-constraints} {https://github.com/cristianvega1995/L1-TV-Experiments-of-Fast-iterative-regularization-by-reusing-data-constraints} 

\subsection{$\ell^1$-norm regularization}
In this section, we apply the routines \ref{A: PDSP} and \ref{A: DPSP}  when  $J$ is equal to the $\ell^1$-norm. We compare the results given by our method with two state-of-the-art regularization procedures: iterative regularization by vanilla primal-dual \cite{molinari2021iterative},  and Tikhonov explicit regularization, using the forward-backward algorithm \cite{Combettes_Wajs2005}. In addition, we compare to another classical optimization algorithm for the minimization of the sum of two non-differentiable functions, namely Douglas-Rachford \cite{briceno2012douglas}. In the noise free case, this algorithm is very effective in terms of number of iterations, but at each iteration it requires the explicit projection on the feasible set. In the noisy case, a stability analysis of the previous is not available. 

We use the four variants of the algorithm \ref{A: PDSP} corresponding to the different choices of the operators $T$ in Definition \ref{O: Operators} and the version of \ref{A: DPSP} described in Example~\ref{e:dpl1}. Unless otherwise stated, in all the experiments we use as preconditioners $\pSigma=\Gamma=\frac{0.99}{\|A\|} \Id$, which both satisfy \eqref{c: ConditionL D1}. 

Let $d=2260$, $p=3000$, and let $A\in\RR^{d\times p}$ be such that every entry of the matrix is an independent sample from $\mathcal{N}(0,1)$, then normalized column by column. We set $b:=Ax^{*}$, where $x^{*}\in \RR^{p}$ is a sparse vector with approximately $300$ nonzero entries uniformly distributed in the interval $[0,1]$. It follows from \cite[Theorem 9.18]{foucart2013invitation} that $x^{*}$ is the unique minimizer of the problem with probability bigger than $0.99$. Let  $b^\delta$ be such that $b^\delta=b+\|b\| u $ where the vector $u$ is distributed, entry-wise, as $U[-0.2,0.2]$.
In this experiment, to test the reconstruction capabilities of our method,  we use the exact datum $x_*$ to establish the best stopping time, i.e. the one minimizing $\|x_k-x_*\|$. The exact solution is also used for the other regularization techniques.  In a real practical situation, if $\delta$ is unknown, we would need to use parameter tuning techniques in order to select the optimal stopping time, but we do not address this aspect here.  

We detail the used algorithms and their parameters below.
\begin{itemize}
    \item[(Tik)] \textbf{Tikhonov Regularization}: We consider a grid of penalty parameters $$G=\left\{\left(1-\frac{l-1}{5}\right)10^{1-d}\|Ab^\delta\|_{\infty} : \ \ l\in[5], \ d\in [6]\right\}$$
    and, for each value $\lambda\in G$, the optimization problem
    \begin{equation} \label{ProbTyk}
    \min\limits_{x\in\RR^{p}}~\left\{\lambda\|x\|_{1}~+~\frac{1}{2}\|Ax-b^\delta\|^2\right\}.
    \end{equation}
    We solve each one of the previous problems with $300$ iterations of forward-backward algorithm, unless the stopping criterion $\|x^{k+1}-x^{k}\|\leq 10^{-3}$ is satisfied earlier. Moreover, to deal efficiently with the sequence of problems, we use warm restart \cite{becker2011nesta}. We first solve problem \eqref{ProbTyk} for the biggest value of $\lambda$ in $G$. Then, we initialize the algorithm for the next value of $\lambda$, in decreasing order, with the solution reached for the previous one; and so on.
    \item[(DR)] \textbf{Douglas Rachford}: see \cite[Theorem 3.1]{briceno2012douglas}.
    \item[(PD)] \textbf{Primal-dual}: this corresponds to  PDA with $m=1$ and $T_1=\Id$.
    \item[(PDS)] \textbf{Primal-dual with serial projections}: at every iteration, we compute a serial projection using all the equations of the noisy system, where the order of the projections is given by a random shuffle. 
    \item[(PDP)]\textbf{Primal-dual with parallel projections}:  $m=1$ and $T_1x=x-\frac{1}{\|A\|_{F}^2}A^{*}(Ax-b^{\delta})$, see Remark \ref{R: Parallel-Landweber}.
    \item[(PDL)]\textbf{Primal-dual Landweber}:  $m=1$ and $T_1x=x-\frac{2}{\|A\|^2}A^{*}(Ax-b^{\delta})$.
    \item[(PDAL)] \textbf{Primal-dual Landweber with adaptive step}: $m=1$, and  $T_1x~=~x~-\beta(x)A^{*}~(Ax~-~b^{\delta})$, where $\beta(x)=\min\left(\frac{\|Ax-b^\delta\|^2}{\|A^{*}(Ax-b^\delta)\|^2}, M\right)$ for $M=10^{6}$. 
     \item[(DPS)]\textbf{Dual primal with serial projections}: at every iteration,  we compute a serial projection  over every inequality of $|A^{*}u|_{\infty}\leq 1$, where the order is given by a random shuffle of the rows of $A^*$.
\end{itemize}
\begin{table}[ht!]\begin{center}
\begin{tabular}{|l|l|l|l|}
\hline
 & Time [S] & Iteration &  \begin{tabular}[c]{@{}l@{}}Reconstruction \\ error\end{tabular} \\ \hline
Tik  & 1.89 & 109 & 3.07  \\ \hline
DR   & 3.08 & 5  & 5.01 \\ \hline
PD   & 0.36 & 14 & 3.11 \\ \hline
PDS  & 1.41 & 11 & 2.58 \\ \hline
PDP  & 0.35 & 14 & 3.11 \\ \hline
PDL  & \textcolor{red}{0.28} & 12 & \textcolor{red}{2.60} \\ \hline
PDAL & \textcolor{red}{0.27} & 11 & \textcolor{red}{2.56} \\ \hline
DPS  & 0.54 & 17 & 2.83 \\ \hline

\end{tabular}\caption{Run-time and number of iterations of each method until it reaches its reconstruction error. We compare the proposed algorithms with Tikhonov regularization (Tik), Douglas-Rachford (DR), and iterative regularization (PD).}
\label{table:1}
\end{center}
\end{table}
\begin{figure}[ht]
    \centering
    \includegraphics[scale=0.25]{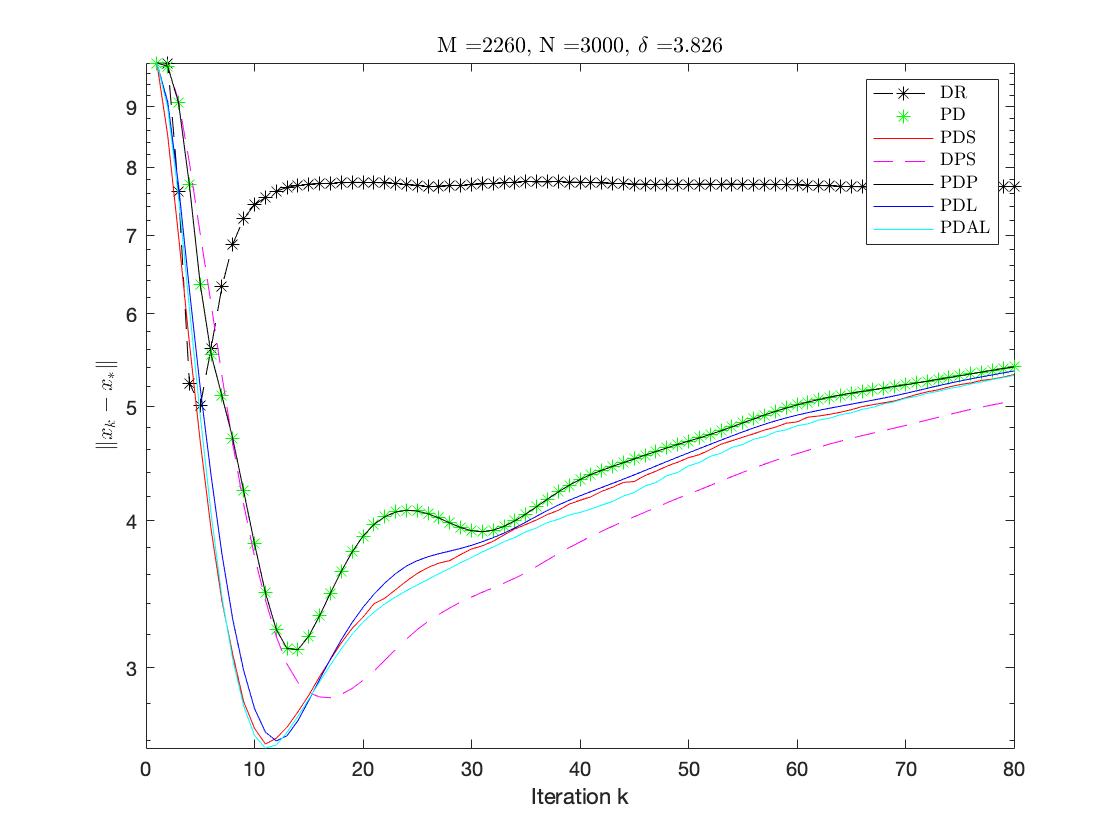}
    \caption{ Graphical representation of early stopping. Note that the first iterates are closer to the noise free solution, then converges to the noisy solution.}
    \label{fig: Early stopping}
\end{figure} \begin{figure}[ht]
    \centering
    \includegraphics[scale=0.25]{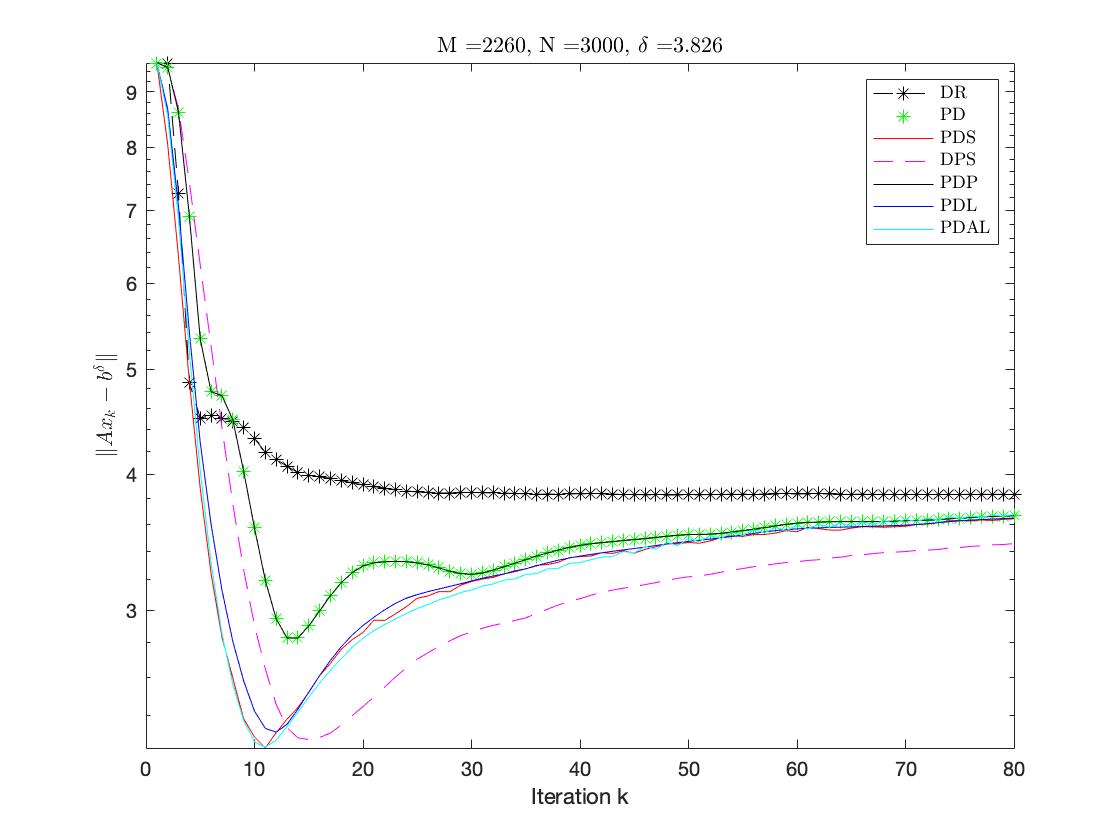}
    \caption{ Early stopping with respect the feasibility. Note that they are similar with respect to previous Figure.}
    \label{fig: Early stoppingFeas}
\end{figure} 
\begin{figure}[ht]
    \centering
    \includegraphics[scale=0.25]{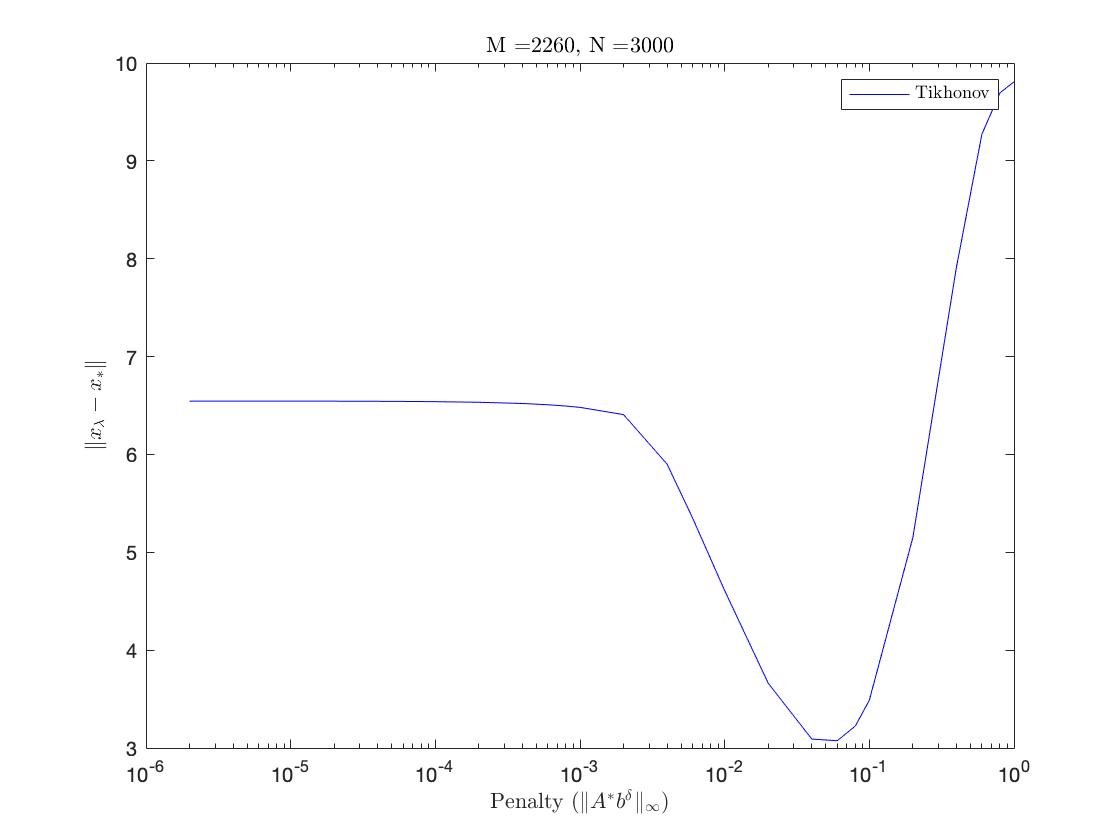}
    \caption{ Reconstruction error of Tikhonov Method with different penalties.}
    \label{fig: Tik}
\end{figure} 
In Table \ref{table:1}, we reported also the number of iterations needed to achieve the best reconstruction error, but it is important to note that the iteration of each method has a different computational cost, so the run-time is a more appropriate comparison criterion.

Douglas-Rachford with early stopping is the regularization method performing worst on this example, both in terms of time and reconstruction error. This behavior may be explained by the fact that this algorithm converges fast convergence to the noisy solution, from which we infer that Douglas-Rachford is not a good algorithm for iterative regularization. Moreover, since  we project on the noisy feasible set at every iteration, the resolution of a linear system is needed at every step. This explains also the cost of each iteration in terms of time. Note in addition that in our example $b^\delta$ is in the range of $A$ and so the noisy feasible set is nonempty. Tikhonov regularization  performs similarly in terms of time, but it requires many more (cheaper) iterations. The achieved error is smaller than the one of DR, but bigger then the minimal one achieved by other methods. 

Regarding our proposals, we observe that the proposed methods perform better than (PD). This supports the idea of reusing the data constraints is beneficial with respect to vanilla primal-dual. The benefit is not evident for (PDP), which achieves the worst reconstruction error, since $\|A\|^2_F$ is very big and so $T_1$ is very close to the identity. All the other methods give better  results in terms of reconstruction error. (PDS) is the slowest since it requires computing several projections at each iteration in a serial manner. We also observe that (PDL) and (PDAL) have better performance improving 22.2\% and 25.0\% in reconstruction error  and  16.4\% and 17.7\% in run-time.

 Figure \ref{fig: Early stopping}  empirically shows the existence of the trade-off between convergence and stability for all the algorithms, and therefore the advantage of early stopping. Similar results were obtained for the feasibility gap.

\subsection{Total variation}
In this section, we perform several numerical experiments using the proposed algorithms for image denoising and deblurring. As done in  the classical image denoising method introduced by Rudin, Osher and Fantemi in \cite{rudin1992nonlinear}, we rely on the total variation regularizer. See also \cite{rudin1992nonlinear,osher1990feature,rudin1994total,chambolle2004algorithm,chambolle1997image,osher2005iterative,xiao2010dual}. We compare (PD) with (PDL) and (PDAL) algorithms, which were the algorithms performing the best in the previous application. In this section, we use two different preconditioners, which have been proved to be very efficient in practice \cite{pock2011diagonal}.

Let $x^{*} \in \mathbb{R}^{N^2}$ represent an image with $N\times N$ pixels in $[0,1]$. We want to recover $x^{*}$  from a blurry and noisy measurement  $y$, i.e. from
\begin{align}
    y=Kx^{*}+e,
\end{align} 
where $K$ is a linear bounded blurring operator and  $e$ is a random noise vector. A standard approach is to assume that the original image is well approximated by the solution of the following constrained minimization problem:
\begin{align}
\label{D:ROF}
\tag{TV}
    \min\limits_{u\in \RR^{N\times N}}&\|Du\|_{1,2}\nonumber\\\text{s.t.}&\hspace{2mm}Ku=y\nonumber,
\end{align}
In the previous,
\begin{align}
    \|\cdot\|_{1,2}\colon  (\RR^{2})^{N\times N}\rightarrow \RR\colon p\rightarrow \sum_{i=1}^{N}\sum_{j=1}^{N}\|p_{ij}\|, 
\end{align}
and $D\colon \RR^{N^2}\rightarrow (\RR^{2})^{N^2}$ is the discrete gradient operator for images, which is defined as
\begin{align}
    \left(D u\right)_{ij}=&\left((D_{x}u)_{ij},(D_{y}u)_{ij}\right)
\end{align}
with \begin{align}
    \left(D_y u\right)_{ij}= &\left\{ \begin{array}{cc}
       u_{i+1,j}-u_{i,j}  & \text{if } 1\leq i\leq N-1 \\
         0 &  \text{if } i=N
    \end{array}\right.\nonumber\\
    \left(D_x u\right)_{ij}=&\left\{ \begin{array}{cc}
       u_{i,j+1}-u_{i,j}  & \text{if } 1\leq j\leq N-1 \\
         0 &  \text{if } j=N.
    \end{array}\right.\nonumber
\end{align}
In order to avoid the computation of   the proximity operator  of $\| D \cdot\|_{1,2} $, we introduce an auxiliary variable $v=Du \in Y:=\RR^{2N^2}$. Since the value in each pixel must belong to $[0,1]$, 
we add the constraint $u\in X:=[0,1]^{N^2}$. In this way, \eqref{D:ROF} becomes
\begin{align}
 \label{D:ROFL1}
\tag{TV}
    \min\limits_{(u,v)\in X\times Y}&\|v\|_{1,2}\nonumber\\\text{s.t.}&\hspace{2mm}Ku=y\nonumber\\ &\hspace{2mm}Du=v\nonumber. 
\end{align}

\subsubsection{Formulation and Algorithms}
Problem~\eqref{D:ROFL1} is a special instance of \eqref{P: problem}, with
\begin{align}
\left\{\begin{array}{l}
     J\colon \RR^{N^2}\times \RR^{2N^2}\mapsto \RR\cup\{+\infty\}\colon x:=(u,v)\mapsto \|v\|_{1,2}+\iota_{X}(u),  \\ \\
     A=\left[\begin{array}{cc}
        K & 0 \\
        D & -\Id
    \end{array}\right],\hspace{2mm}  b^{\delta}=\left[\begin{array}{c}
        y \\
        0 
    \end{array}\right], \text{ and }  p=d=3N^2.
\end{array}
 \right.
    \label{S: TV}
\end{align} 
Clearly, $A$ is a linear bounded nonzero operator, and $J\in\Gamma_0(\RR^{N^2}\times \RR^{2N^2})$.
\begin{table}[ht]
    \centering
\begin{tabular}{|m{120mm}|}
          \hline   Primal-Dual for total variation  \\  \hline\vspace{2mm} \textbf{Input}:  $(\proj^0,\proj^{-1},\prim^{0},v^0)\in\RR^{6N^2}\times\RR^{2N^2}$ and $(q^0,q^{-1},z^{0},w^0)\in\RR^{3N^2}\times\RR^{N^2}$.\\\vspace{2mm} \begin{flushleft}\textbf{For} $k=1,\ldots,\text{N:}$\end{flushleft}\vspace{-4mm}\\ \vspace{-10mm}
        \begin{align}
    \begin{array}{l}
v^{k+1}= v^k+\Gamma( K(\proj^{k}+ \prim^{k}-\proj^{k-1})^k-y)\\
w^{k+1}= w^k-\Gamma(q^{k}+ z^{k}-q^{k-1})+\Gamma D(\proj^{k}+ \prim^{k}-\proj^{k-1})\\
\prim^{k+1}=P_{X}(\proj^k-\pSigma K^*v^{k+1}+\pSigma w^{k+1})\\ z^{k+1}=\prox_{\pSigma \|\cdot\|_{1,2}}(q^k-\pSigma D^*w^{k+1})\\
\proj^{k+1}=x^{k}-\alpha(x^{k})\left (K^{*}(Kx^{k}-y)+(Dx^{k}-z^{k})\right)\\q^{k+1}=q^{k}-\alpha(x^{k})D^{*}\left(Dx^{k}-z^{k}\right)
\end{array}  \label{A: PDA-TV}\end{align}\vspace{-4mm}\\ \vspace{0mm}\begin{flushleft}\textbf{End}\end{flushleft}  \\  \hline
      \end{tabular}    
 \caption{General form of the algorithms.}
     \label{tab: Tabla 1}
\end{table}

We compare the algorithms listed below. 
Note that all the proposed algorithms are different instances of the general routine described in Table~\ref{tab: Tabla 1}, and each one of them corresponds to a different choice of $\alpha(x^k)$:
\begin{enumerate}
    \item  PD, the vanilla primal-dual algorithm, corresponding to $\alpha(x^k)=0$;
    \item PPD, the preconditioned primal-dual algorithm, obtained by $\alpha(x^k)=0$ and $\pSigma$ and $\Gamma$ as in \cite[Lemma 2]{pock2011diagonal};
     \item PDL, corresponding to $\alpha(x^k)=1/\|A\|^2$;
      \item PDAL, corresponding to $\alpha(x^k)=\beta(x^k)$ as \eqref{d: averaged4}.
\end{enumerate}
Initializing by $p^0=\bar{p}^{0}=x^{0}$ and  $q^0=\bar{q}^{0}=z^{0}$, we recover the results of Theorem \ref{Th:PPD} and Corollary \ref{ESPDA}.
\begin{remark}
     In order to implement the algorithm in \ref{A: PDA-TV}, we first need to compute the
following operators.
\begin{enumerate}
    \item It follows from \cite[Proposition 24.11]{bauschke2011convex} and  \cite[Example 24.20]{bauschke2011convex} that
   $$\prox^{\pSigma}_{\|\cdot\|_{1,2}}(v)=\left(\prox^{\pSigma_{i}}_{\|\cdot\|}(v_{i})\right)_{i=1}^{N^2}=\left(\left(1-\frac{\pSigma}{\max\{\pSigma,\|v\|\}}\right)v_{i}\right)_{i=1}^{N^2},$$
    where $v_i\in\RR^{2}$.
    Analogously, the projection onto $X$ can be computed as
   $$P_{X}(u)=\left(P_{[0,1]}(u_i)\right)_{i=1}^{N^2},$$
    where $P_{[0,1]}(u_i)=\min\{1,\max\{u_i,0\} \}.$
\item It follows from \cite{chambolle2004algorithm} that 
\begin{align}
  \hspace{-6mm}-D^{*}p= \operatorname{div}p=\left\{ \begin{array}{ll}
       (p_1)_{i,j}-(p_1)_{i-1,j}  & \text{if } 1< i<N \\  (p_1)_{i,j}  & \text{if } i=1\\
         -(p_1)_{i-1,j} &  \text{if } i=N
    \end{array}\right.\hspace{-2mm}+\left\{ \begin{array}{ll}
       (p_2)_{i,j}-(p_2)_{i,j-1}  & \text{if } 1< j<N \\  (p_2)_{i,j}  & \text{if } j=1\\
         -(p_2)_{i,j-1} &  \text{if } j=N.
    \end{array}\right.\nonumber
\end{align}
\end{enumerate}
\end{remark}
\subsubsection{Numerical results} Set $N=256$, let $x^{*}$ be the image \textquotedblleft boat\textquotedblright\ in the library Numerical tours \cite{peyre2011numerical}. We suppose that $K$ is an operator assigning to every pixel the average of the pixels in a neighborhood of radius 8 and that $e\thicksim U[-0.025,0.025]^{N^2}$. We use the original image as exact solution. For denoising and deblurring, we early stop the procedure at the iteration minimizing the mean square error (MSE), namely $\|x^k-x^*\|^2/N^2$, and we measure the time and the number of iterations needed to reach it. Another option for early stopping could be to consider the image with minimal structural similarity (SSIM). Numerically, in our experiments, this gives the same results. Additionally, we use the peak signal-to-noise ratio (PSNR) to compare the images.
Note that primal-dual algorithm with preconditioning is the method that needs less time and iterations among all the procedures. Moreover, due to \cite[Lemma 2]{chambolle2011first}, condition \eqref{c: ConditionL D1} is automatically satisfied, while for the other methods we need to check it explicitly, which is computationally costly. However, (PPD) is the worst in terms of SSIM, PNSR, and MSE. We verify that all other algorithms have a superior performance in terms of reconstruction, with a small advantage for the Landweber with fixed and adaptive step-sizes, reducing the MSE of $94\%$ with respect the noisy image. In addition, compared to (PD), (PDL) and (PDAL) require less iterations and time to satisfy the early stopping criterion. We believe that this is due to the fact that the extra Landweber operator improves the feasibility of the primal iterates. Visual assessment of the denoised and deblurred images are shown in Figure \ref{fig: Comparision_TV}, that highlights the regularization properties achieved by the addition of the Landweber operator and confirms  the previous conclusions.
 \begin{table}[ht]\centering
\resizebox{0.75\textwidth}{!}{%
\begin{tabular}{|l|l|l|l|l|l|}
\hline
                   & Iterations & Time   & SSIM   & PNSR    & MSE    \\ \hline
Noisy image        & -          & -      & 0.4468 & 21.4801 & 0.0071 \\ \hline
PD       & 54         & 8.9773 & 0.8928 & 32.3614 & 0.0006 \\ \hline
\begin{tabular}[c]{@{}l@{}}PD with\\ preconditioning\end{tabular} & 5 & 1.5515 & 0.8581 & 27.3753 & 0.0018 \\ \hline
PDL         & 46         & 7.1846 & 0.9066 & 34.2174 & 0.0004 \\ \hline
PDAL & 31         & 5.4542 & 0.9112 & 34.3539 & 0.0004 \\ \hline
\end{tabular}%
} \caption{Quantitative comparison of the algorithms in terms of Structural similarity (SSIM), peak signal-to-noise ratio (PSNR), Mean square error (MSE), time, and iterations to reach the early stopping.}
     \label{tab: Tabla 2}
\end{table}
\begin{figure}[ht]
    \centering
    \includegraphics[scale=0.40]{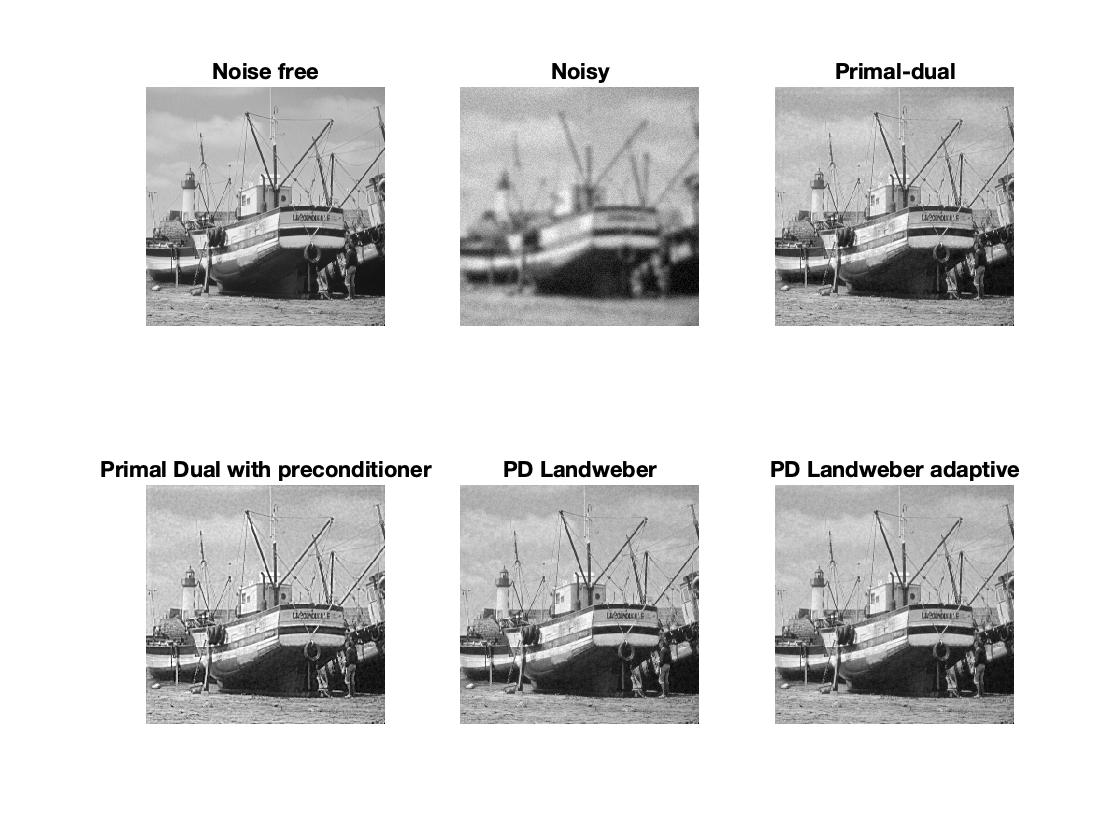}
    \caption{ Qualitative comparison of the 4 proposed methods.}
    \label{fig: Comparision_TV}
\end{figure}
\section{Conclusion and Future Work}
\label{s: Conclusion}
In this paper we studied two  new  iterative  regularization methods for solving a linearly constrained minimization problem, based on an extra activation step reusing the the data constraint. The analysis was carried out in the context of convex functions and  worst-case deterministic noise.  We proposed five instances of our  algorithm and compared their numerical performance with state of the art methods and we observed considerable improvement in run-time.

In the future, we would like to extend Theorem~\ref{Th:PPD} to structured convex problems and more specific algorithms. Possible extensions are:
1) the study of problems including, in the objective function, a $L$-smooth term and a composite linear term; 2) the analysis of random updates in the dual variable (see \cite{chambolle2018stochastic}) and stochastic approximations for the gradient; 3) the theoretical study of the impact of different preconditioners; 4) the improvement of the convergence and stability rates for strongly convex objective functions.
\section{Acknowledgement}
This project has been supported by TraDE-OPT project which received funding from the European Union’s Horizon 2020 research and innovation program under the Marie Skłodowska-Curie grant agreement No 861137. L.R. acknowledges support from the Center for Brains, Minds and Machines (CBMM), funded by NSF STC award CCF-1231216. L.R. also acknowledges the financial support of the European Research Council (grant SLING 819789), the AFOSR projects FA9550-18-1-7009, FA9550-17-1-0390 and BAA-AFRL-AFOSR-2016-0007 (European Office of Aerospace Research and Development), and the EU H2020-MSCA-RISE project NoMADS - DLV-777826. C. M. e S. V. are members of the INDAM-GNAMPA research group.  This work represents only the view of the authors. The European Commission and the other organizations are not responsible for any use that may be made of the information it contains.

\section{Proofs}
\subsection{Equivalence between Primal-dual and Dual-primal algorithms.}\label{Proof:PD=Dp}
In the following lemma we establish that, if $T=\Id$ and the initialization is the same, then there is an equivalence between the $k$-th primal variable of \ref{A: PDSP} and \ref{A: DPSP}, denoted by $\dal^{k}_{PD}$ and $\dal^{k}_{DP}$, respectively.

\begin{lemma}
\label{L: PD=DP}
Let  $(\proj^0_{PD},\bar{\proj}^{0}_{PD},\dal^0_{PD})\in \RR^{2p}\times\RR^{d}$ and $(\prim_{DP}^{0},\prop_{DP}^0,\bar{\prop}_{DP}^{0})\in \RR^{p}\times\RR^{2d}$ the initialization \ref{A: PDSP} and \ref{A: DPSP}, respectively, in the case when $m=1$ and $T=\Id$. Suppose that $\proj_{PD}^0=\bar{\proj}_{PD}^{0}$, $\prop_{DP}^0=\bar{\prop}_{DP}^{0}$, $\dal_{P    D}^0=\prop_{DP}^{0}$, and $\prim_{PD}^{1}=\prim_{DP}^{1}$, then for every $k\in\NN$, $\prim^{k}_{PD}=\prim^{k}_{DP}$.
\end{lemma}
\begin{proof}
Since $m=1$ and $T=\Id$ in both algorithms, for every $k\in\NN$, yields $\prim_{PD}^{k}=\proj_{PD}^{k}$ and $\dal_{DP}^{k}=\prop_{DP}^{k}$. On one hand, by definition of \ref{A: PDSP}, we have that
 \begin{align}
     \dal^{k+1}_{PD}&=\dal^{1}_{PD}+\Gamma\sum_{i=1}^{k}\left(A\bar{\proj}^{i}_{PD}-b^{\delta}\right)\nonumber\\&=\dal^{1}_{PD}+\sum_{i=1}^{k}\Gamma A(\proj^{i}_{PD}-\proj^{i-1}_{PD})+\Gamma\sum_{i=1}^{k}\left(A\prim^{i}_{PD}-b^{\delta}\right)\nonumber\\&=\dal^{1}_{PD}+\Gamma A(\proj^{k}_{PD}-\proj^{0}_{PD})+\Gamma\sum_{i=1}^{k}\left(A\prim^{i}_{PD}-b^{\delta}\right)\nonumber\\&=\dal^{0}_{PD}+\Gamma(A\prim^{k}_{PD}-b^\delta)+\Gamma\sum_{i=1}^{k}\left(A\prim^{i}_{PD}-b^{\delta}\right),\label{V: Uprimal}
 \end{align}
 where in the last equality is obtained since $\proj_{PD}^0=\bar{\proj}_{PD}^{0}$. 
 Replacing \eqref{V: Uprimal} in the definition of $x^{k+1}_{PD}$
 \begin{align}
   \prim_{PD}^{k+1}= \prox^{\pSigma }_{J}\left(\prim_{PD}^{k}-\pSigma A^{*}\left(\dal_{PD}^{0}+\Gamma (A\prim_{PD}^{k}-b^\delta)+\Gamma\sum_{i=1}^{k}\left(A\prim_{PD}^{i}-b^{\delta}\right)\right)\right).
     \label{e: algonestep}
 \end{align}
 On the other hand, by \ref{A: DPSP} we have that
 \begin{align}
     \dal^{k+1}_{DP}=\prop^{k+1}_{DP}=\prop^{0}_{DP}+\Gamma\sum_{i=1}^{k+1}\left(A\prim^{i}_{DP}-b^{\delta}\right),\label{V: Udual}
 \end{align}
and 
\begin{align}
    \bar{\prop}^{k}_{DP}=\prop^{k}_{DP}+\dal^{k}_{DP}-\prop^{k-1}_{DP}=\prop^{0}_{DP}+\Gamma(A\prim^{k}_{DP}-b^\delta)+\Gamma\sum_{i=1}^{k}\left(A\prim^{i}_{DP}-b^{\delta}\right),\label{V: UbarDual}.
\end{align}
Replacing \eqref{V: UbarDual} in \ref{A: DPSP}, for every $k>1$, we can deduce that
 \begin{align}
   \prim_{DP}^{k+1}= \prox^{\pSigma }_{J}\left(\prim_{DP}^{k}-\pSigma A^{*}\left(\prop_{DP}^{0}+\Gamma (A\prim_{DP}^{k}-b^\delta)+\Gamma\sum_{i=1}^{k}\left(A\prim_{DP}^{i}-b^{\delta}\right)\right)\right).
     \label{e: algonestep1}
 \end{align}
Since $\dal^{0}_{PD}=\prop^{0}_{DP}$ and $\prim^{1}_{PD}=\prim^{1}_{DP}$ the result follows by induction.
  \end{proof}
   \begin{remark}
An analysis similar to that in the proof of Lemma \ref{L: PD=DP} shows that
 \begin{align}
   \prim_{PD}^{k+1}= \prox^{\pSigma }_{J}\left(\prim^{k}_{PD}-\pSigma A^{*}\left(\dal^{0}_{PD}+\Gamma (A
     T_{\epsilon_{k}}\prim^{k}_{PD}-b^\delta)+\Gamma\sum_{i=1}^{k}\left(A\prim^{i}_{PD}-b^{\delta}\right)\right)\right),
     \label{e: algonestepproj}
 \end{align}
 which implies that the algorithm can be written in one step if we only care about the primal variable. 
 \end{remark}
\subsection{Proof of Theorem \ref{Th:PPD}}\label{Proof:PPD}
\begin{proof}
From \ref{A: PDSP}, we deduce that: 
\begin{align}
    \pSigma^{-1}(\proj^k-\prim^{k+1})- A^*\dal^{k+1}&\in\partial J(\prim^{k+1})\nonumber\\
   \Gamma^{-1}(\dal^k-\dal^{k+1}) +A\bar{\proj}^k &=b^{\delta}\label{e:PMI}
\end{align}
Therefore, we have
\begin{align}
   \left(\forall x\in\RR^p\right)\hspace{3mm} J(\prim^{k+1})+\scal{\pSigma^{-1}(\proj^k-\prim^{k+1})- A^*\dal^{k+1}}{\prim-\prim^{k+1}}\leq J(\prim)
    \label{e:subdif P}
\end{align} and \eqref{e:subdif P} yields
\begin{align}
         0\geq& J(\prim^{k+1})-J(\prim)+\scal{\pSigma^{-1}(\proj^k-\prim^{k+1})-A^*\dal^{k+1}}{\prim-\prim^{k+1}}\nonumber\\=& J(\prim^{k+1})-J(\prim)+\frac{\|\proj^k-\prim^{k+1}\|_{\pSigma^{-1}}^2}{2}+\frac{\|\prim^{k+1}-\prim\|_{\pSigma^{-1}}^2}{2}\nonumber\\&-\frac{\|\proj^k-\prim\|_{\pSigma^{-1}}^2}{2}+\scal{\prim^{k+1}-\prim}{A^*\dal^{k+1}} 
                 \label{e: psub P}
\end{align}
Analogously by \eqref{e:PMI} we get
\begin{align}
    0=&\scal{\Gamma^{-1}(\dal^k-\dal^{k+1})+ A\bar{\proj}^k-b^\delta}{\dal-\dal^{k+1}}\nonumber\\0=& \frac{\|\dal^{k+1}-\dal^{k}\|_{\Gamma^{-1}}^2}{2}+\frac{\|\dal^{k+1}-\dal\|_{\Gamma^{-1}}^2}{2}-\frac{\|\dal^{k}-\dal\|_{\Gamma^{-1}}^2}{2}+\scal{b^{\delta}-A\bar{\proj}^k}{\dal^{k+1}-\dal}\label{e: dsub P}
\end{align}
Recall that $z:=(\prim,\dal)\in\mathcal{Z}\subset C\times \RR^{d}$, $z^{k}:=(\prim^{k},\dal^{k})$, and $V(z):=\frac{\|\prim\|^2_{\pSigma^{-1}}}{2}+\frac{\|\dal\|_{\Gamma^{-1}}^2}{2}$. Summing \eqref{e: psub P} and  \eqref{e: dsub P}, and by Assumption  A3, we obtain \\
\begin{align}
          J(\prim^{k+1})-J(\prim)+\frac{\|\prim^{k+1}-\proj^k\|_{\pSigma^{-1}}^2}{2} +\frac{\|\dal^{k+1}-\dal^{k}\|_{\Gamma^{-1}}^2}{2}+V(z^{k+1}-z)-V(z^{k}-z)&\nonumber\\+\scal{A(\prim^{k+1}-\prim)}{\dal^{k+1}}+\scal{b^\delta-A\bar{\proj}^k}{\dal^{k+1}-\dal}-\frac{e\delta^2}{2}&\leq 0\label{e: prestimate P}
    \end{align}
    Now compute 
    \begin{align}
        & J(\prim^{k+1})-J(\prim)+\scal{A(\prim^{k+1}-\prim)}{\dal^{k+1}} +\scal{b^{\delta}-A\bar{\proj}^k}{\dal^{k+1}-\dal}\nonumber\\ =&\mathcal{L}(\prim^{k+1},\dal)-\mathcal{L}(\prim,\dal^{k+1})-\scal{A\prim^{k+1}-b}{\dal}+\scal{A\prim-b}{\dal^{k+1}}\nonumber\\&+\scal{A(\prim^{k+1}-\prim)}{\dal^{k+1}} +\scal{b^{\delta}-A\bar{\proj}^k}{\dal^{k+1}-\dal}\nonumber\\ =&\mathcal{L}(\prim^{k+1},\dal)-\mathcal{L}(\prim,\dal^{k+1})-\scal{A\prim^{k+1}}{\dal}+\scal{b}{\dal}+\scal{A\prim}{\dal^{k+1}}-\scal{b}{\dal^{k+1}}\nonumber\\&+\scal{A\prim^{k+1}}{\dal^{k+1}}-\scal{A\prim}{\dal^{k+1}} +\scal{b^\delta}{\dal^{k+1}-\dal}-\scal{A\bar{\proj}^k}{\dal^{k+1}-\dal}\nonumber\\ =&\mathcal{L}(\prim^{k+1},\dal)-\mathcal{L}(\prim,\dal^{k+1})+\scal{b^{\delta}-b}{\dal^{k+1}-\dal} +\scal{A\prim^{k+1}-A\bar{\proj}^k}{\dal^{k+1}-\dal}\nonumber\\ \geq&\mathcal{L}(\prim^{k+1},\dal)-\mathcal{L}(\prim,\dal^{k+1})-\delta\|\Gamma^{\frac{1}{2}}\|\|\dal^{k+1}-\dal\|_{\Gamma^{-1}} +\scal{A\prim^{k+1}-A\bar{\proj}^k}{\dal^{k+1}-\dal}.
        \label{e: lagrange P}
    \end{align}
   From \eqref{e: lagrange P} and \eqref{e: prestimate P} we obtain
\begin{align}
         &\mathcal{L}(\prim^{k+1},\dal)-\mathcal{L}(\prim,\dal^{k+1})+\frac{\|\prim^{k+1}-\proj^k\|_{\pSigma^{-1}}^2}{2} +\frac{\|\dal^{k+1}-\dal^{k}\|_{\Gamma^{-1}}^2}{2}\nonumber\\&+V(z^{k+1}-z)-V(z^{k}-z)-\delta\|\Gamma^{\frac{1}{2}}\|\|\dal^{k+1}-\dal\|_{\Gamma^{-1}}-\frac{e\delta^2}{2}\nonumber\\ \leq& -\scal{A(\prim^{k+1}-\bar{\proj}^k)}{\dal^{k+1}-\dal}\label{e: lagrangeI.25 P} \\ = & -\scal{A(\prim^{k+1}-\proj^k)}{\dal^{k+1}-\dal}+\scal{A(\prim^{k}-\proj^{k-1})}{\dal^{k}-\dal}\nonumber\\ &+\scal{A(\prim^{k}-\proj^{k-1})}{\dal^{k+1}-\dal^{k}}\label{e: lagrangeI.5 P} \\ = & -\scal{A(\prim^{k+1}-\proj^k)}{\dal^{k+1}-\dal}+\scal{A(\prim^{k}-\proj^{k-1})}{\dal^{k}-\dal}\nonumber\\ &+\scal{\gamedio A\pSigma^{\frac{1}{2}}\pSigma^{-\frac{1}{2}}(\prim^{k}-\proj^{k-1})}{\gamediomenos(\dal^{k+1}-\dal^{k})}\nonumber\\  \leq& -\scal{A(\prim^{k+1}-\proj^k)}{\dal^{k+1}-\dal}+\scal{A(\prim^{k}-\proj^{k-1})}{\dal^{k}-\dal}\nonumber\\ &+\|\gamedio A\pSigma^{\frac{1}{2}}\|^2\frac{\|\dal^{k+1}-\dal^{k}\|_{\Gamma^{-1}}^2}{2}+\frac{\|\prim^{k}-\proj^{k-1}\|_{\pSigma^{-1}}^2}{2}\label{e: lagrangeIP}
    \end{align}
   Then, recalling that  $\alpha=1-\|\gamedio A\pSigma^{\frac{1}{2}}\|^2$,  we have the following estimate
\begin{align}
        &\mathcal{L}(\prim^{k+1},\dal)-\mathcal{L}(\prim,\dal^{k+1})+\frac{\|\prim^{k+1}-\proj^{k}\|_{\pSigma^{-1}}^2}{2}-\frac{\|\prim^{k}-\proj^{k-1}\|_{\pSigma^{-1}}^2}{2}\nonumber\\&+\frac{\alpha}{2}\|\dal^{k+1}-\dal^{k}\|_{\Gamma^{-1}}^2+V(z^{k+1}-z)-V(z^{k}-z)\nonumber\\ \leq& \delta\|\Gamma^{\frac{1}{2}}\|\|\dal^{k+1}-\dal\|_{\Gamma^{-1}} -\scal{A(\prim^{k+1}-\proj^k)}{\dal^{k+1}-\dal}\nonumber\\&+\scal{A(\prim^{k}-\proj^{k-1})}{\dal^{k}-\dal}+\frac{e\delta^2}{2}
        \label{e: lagrangeII P}
    \end{align}
      Summing from $1$ to $N-1$ we obtain 
    \begin{align}
        &\sum_{k=1}^{N-1}\left(\mathcal{L}(\prim^{k+1},\dal)-\mathcal{L}(\prim,\dal^{k+1})\right)+\frac{\|\prim^{N}-\proj^{N-1}\|_{\pSigma^{-1}}^2}{2}-\frac{\|\prim^{1}-\proj^{0}\|_{\pSigma^{-1}}^2}{2}\nonumber\\&+\frac{\alpha}{2}\sum_{k=1}^{N-1}\|\dal^{k+1}-\dal^{k}\|_{\Gamma^{-1}}^2+V(z^{N}-z)-V(z^{1}-z)-\scal{ A(\prim^{1}-\proj^{0})}{\dal^{1}-\dal} \nonumber\\\leq& \delta\|\Gamma^{\frac{1}{2}}\|\sum_{k=1}^{N}\|\dal^{k}-\dal\|_{\Gamma^{-1}}-\scal{\gamedio A\pSigma^{\frac{1}{2}}\pSigma^{-\frac{1}{2}}(\prim^{N}-\proj^{N-1})}{\gamediomenos(\dal^{N}-\dal)}+\frac{(N-1) e\delta^2}{2}\nonumber\\\leq&\delta\|\Gamma^{\frac{1}{2}}\|\sum_{k=1}^{N}\|\dal^{k}-\dal\|_{\Gamma^{-1}}+\frac{\|\prim^{N}-\proj^{N-1}\|_{\pSigma^{-1}}^2}{2}+\|\gamedio A\pSigma ^{\frac{1}{2}}\|^2\frac{\|\dal^{N}-\dal\|_{\Gamma^{-1}}^2}{2}+\frac{(N-1) e\delta^2}{2}
        \label{e: lagrangeIII.5 P}
    \end{align}
    Now, by choosing $k=1$ in \eqref{e: lagrangeI.25 P} we get
    \begin{align}
         &\mathcal{L}(\prim^{1},\dal)-\mathcal{L}(\prim,\dal^{1})+\frac{\|\prim^{1}-\proj^0\|_{\pSigma^{-1}}^2}{2} +\frac{\alpha}{2}\|\dal^{1}-\dal^{0}\|_{\Gamma^{-1}}^2\nonumber\\&+V(z^{1}-z)-V(z^{0}-z) +\scal{A(\prim^{1}-\bar{\proj}^0)}{\dal^{1}-\dal}\nonumber\\ \leq&\delta\|\Gamma^{\frac{1}{2}}\|\|\dal^{1}-\dal\|_{\Gamma^{-1}}+\frac{e\delta^2}{2}.
          \label{e: lagrangeIII.5 P1}
         \end{align}
   Adding \eqref{e: lagrangeIII.5 P} and \eqref{e: lagrangeIII.5 P1} we obtain
     \begin{align}
        &\sum_{k=0}^{N-1}\left(\mathcal{L}(\prim^{k+1},\dal)-\mathcal{L}(\prim,\dal^{k+1})\right)+\frac{\alpha}{2}\|\dal^{N}-\dal\|_{\Gamma^{-1}}^2\nonumber\\&+\sum_{k=1}^{N}\frac{\alpha}{2}\|\dal^{k}-\dal^{k-1}\|_{\Gamma^{-1}}^2+\frac{\|\prim^{N}-\prim\|_{\pSigma^{-1}}^2}{2} \nonumber\\\leq& \delta\|\Gamma^{\frac{1}{2}}\|\sum_{k=1}^{N}\|\dal^{k}-\dal\|_{\Gamma^{-1}}+V(z^{0}-z)+\frac{N e\delta^2}{2}
        \label{e: lagrangeIV.5 P}
    \end{align}
Next, by \eqref{e: lagrangeI.5 P} we have the following estimate
\begin{align}
        &\frac{\|\prim^{k+1}-\proj^{k}\|_{\pSigma^{-1}}^2}{2}-\scal{A(\prim^{k}-\proj^{k-1})}{\dal^{k+1}-\dal^k}+\frac{\|\dal^{k+1}-\dal^k\|_{\Gamma^{-1}}^2}{2}\nonumber\\&+\mathcal{L}(\prim^{k+1},\dal)-\mathcal{L}(\prim,\dal^{k+1})+V(z^{k+1}-z)-V(z^{k}-z)\nonumber\\ \leq& \delta\|\Gamma^{\frac{1}{2}}\|\|\dal^{k+1}-\dal\|_{\Gamma^{-1}} -\scal{A(\prim^{k+1}-\proj^{k})}{\dal^{k+1}-\dal} \nonumber\\&+\scal{A(\prim^{k}-\proj^{k-1})}{\dal^{k}-\dal}+\frac{e\delta^2}{2}
        \label{e: lagrangeII.5 P}
    \end{align}
       Summing from $1$ to $N-1$  we obtain 
    \begin{align}
         &\sum_{k=1}^{N-1}\left(\frac{\|\prim^{k+1}-\proj^{k}\|_{\pSigma^{-1}}^2}{2}-\scal{A(\prim^{k}-\proj^{k-1})}{\dal^{k+1}-\dal^{k}} +\frac{\|\dal^{k+1}-\dal^{k}\|_{\Gamma^{-1}}^2}{2}\right)\nonumber\\&+\sum_{k=1}^{N-1}\left(\mathcal{L}(\prim^{k+1},\dal)-\mathcal{L}(\prim,\dal^{k+1})\right)+V(z^{N}-z)-V(z^{1}-z) -\scal{A(\prim^{1}-\proj^{0})}{\dal^{1}-\dal}\nonumber\\
         \leq& \delta\|\Gamma^{\frac{1}{2}}\|\sum_{k=1}^{N-1}\|\dal^{k+1}-\dal\|_{\Gamma^{-1}}  -\scal{A(\prim^{N}-\proj^{N-1})}{\dal^{N}-\dal}+\frac{(N-1) e\delta^2}{2}\nonumber\\
         =& \delta\|\Gamma^{\frac{1}{2}}\|\sum_{k=1}^{N-1}\|\dal^{k+1}-\dal\|_{\Gamma^{-1}}  -\scal{\gamedio A\pSigma^{\frac{1}{2}}\pSigma^{-\frac{1}{2}}(\prim^{N}-\proj^{N-1})}{\gamediomenos(\dal^{N}-\dal)}+\frac{(N-1) e\delta^2}{2}\nonumber\\\leq& \delta\|\Gamma^{\frac{1}{2}}\|\sum_{k=1}^{N-1}\|\dal^{k+1}-\dal\|_{\Gamma^{-1}} +\frac{ \|\gamedio A\pSigma^\frac{1}{2}\|^2}{2}\|\prim^{N}-\proj^{N-1}\|_{\pSigma^{-1}}^2+\frac{\|\dal^{N}-\dal\|_{\Gamma^{-1}}^2}{2}+\frac{(N-1) e\delta^2}{2}
        \label{e: lagrangeIII P}
    \end{align}
    Now, since $\dal^{k+1}-\dal^{k}=\gammmadot\left(A\bar{\proj}^k-b^\delta\right)$ we derive that 
    \begin{align}
     &\sum_{k=1}^{N-1}\left(\frac{\|\prim^{k+1}-\proj^{k}\|_{\pSigma^{-1}}^2}{2}-\scal{A(\prim^{k}-\proj^{k-1})}{\dal^{k+1}-\dal^{k}} +\frac{\|\dal^{k+1}-\dal^{k}\|_{\Gamma^{-1}}^2}{2}\right)\nonumber\\ =&  \sum_{k=1}^{N-1}\left(\frac{\|\prim^{k}-\proj^{k-1}\|_{\pSigma^{-1}}^2}{2}-\scal{A(\prim^{k}-\proj^{k-1})}{\dal^{k+1}-\dal^{k}}+\frac{\|\dal^{k+1}-\dal^{k}\|_{\Gamma^{-1}}^2}{2}\right)\nonumber\\&+ \frac{\|\prim^{N}-\proj^{N-1}\|_{\pSigma^{-1}}^2}{2}-\frac{\|\prim^{1}-\proj^{0}\|_{\pSigma^{-1}}^2}{2}\nonumber\\=&  \sum_{k=1}^{N-1}\left(\frac{\|\gamedio A(\prim^{k}-\proj^{k-1})\|^2}{2}-\scal{\Gamma^{\frac{1}{2}}A(\prim^{k}-\proj^{k-1})}{\Gamma^{\frac{1}{2}}(A\bar{\proj}^k-b^{\delta})}+\frac{\|\gamedio(A\bar{\proj}^k-b^{\delta})\|^2}{2}\right)\nonumber\\
   &+ \frac{\|\prim^{N}-\proj^{N-1}\|_{\pSigma^{-1}}^2}{2}-\frac{\|\prim^{1}-\proj^{0}\|
   ^2}{2}\nonumber\\&+\sum_{k=1}^{N-1}\left(\frac{\|\prim^{k}-\proj^{k-1}\|_{\pSigma^{-1}}^2}{2}-\frac{\|\gamedio A(\prim^{k}-\proj^{k-1})\|^2}{2}\right)\nonumber\\
   =&  \sum_{k=1}^{N-1}\frac{\|\gamedio ( A\proj^{k}-b^{\delta})\|^2}{2}+ \frac{\|\prim^{N}-\proj^{N-1}\|_{\pSigma^{-1}}^2}{2}-\frac{\|\prim^{1}-\proj^{0}\|_{\pSigma^{-1}}^2}{2}\nonumber\\&+\sum_{k=1}^{N-1}\left(\frac{\|\prim^{k}-\proj^{k-1}\|_{\pSigma^{-1}}^2}{2}-\frac{\|\gamedio A(\prim^{k}-\proj^{k-1})\|^2}{2}\right),\end{align}
  and since $\alpha=1-\|\gamedio A\pSigma^{\frac{1}{2}}\|^2>0$ we obtain
   \begin{align}
   & \sum_{k=1}^{N-1}\frac{\|\gamedio ( A\proj^{k}-b^{\delta})\|^2}{2}+ \frac{\|\prim^{N}-\proj^{N-1}\|_{\pSigma^{-1}}^2}{2}-\frac{\|\prim^{1}-\proj^{0}\|_{\pSigma^{-1}}^2}{2}\nonumber\\&+\sum_{k=1}^{N-1}\left(\frac{\|\prim^{k}-\proj^{k-1}\|_{\pSigma^{-1}}^2}{2}-\frac{\|\gamedio A(\prim^{k}-\proj^{k-1})\|^2}{2}\right) \nonumber\\
   \geq&  \sum_{k=1}^{N-1}\frac{\|\gamedio (A\proj^{k}-b^{\delta})\|^2}{2}+ \frac{\|\prim^{N}-\proj^{N-1}\|_{\pSigma^{-1}}^2}{2}-\frac{\|\prim^{1}-\proj^{0}\|_{\pSigma^{-1}}^2}{2}\nonumber\\ &+\frac{\alpha}{2}\sum_{k=1}^{N-1}\|\gamedio A(\prim^{k}-\proj^{k-1})\|^2\nonumber\\
   \geq&  \sum_{k=1}^{N-1}\frac{\|\gamedio ( A\proj^{k}-b^{\delta})\|^2}{2}+ \frac{\|\prim^{N}-\proj^{N-1}\|_{\pSigma^{-1}}^2}{2}-\frac{\|\prim^{1}-\proj^{0}\|_{\pSigma^{-1}}^2}{2}\nonumber\\&-\frac{\alpha}{2}\|\gamedio A(\prim^{N}-\proj^{N-1})\|^2+\frac{\alpha}{2}\|\gamedio A(\prim^{1}-\proj^{0})\|^2\nonumber\\
   &+\frac{\alpha}{2}\sum_{k=1}^{N-1}\|\gamedio A(\prim^{k+1}-\proj^{k})\|^2.\end{align}
  In turn, by convexity of $\|\cdot\|^2$ results in
   \begin{align}
   &  \sum_{k=1}^{N-1}\frac{\|\gamedio ( A\proj^{k}-b^{\delta})\|^2}{2}+ \frac{\|\prim^{N}-\proj^{N-1}\|_{\pSigma^{-1}}^2}{2}-\frac{\|\prim^{1}-\proj^{0}\|_{\pSigma^{-1}}^2}{2}\nonumber\\&-\frac{\alpha}{2}\|\gamedio A(\prim^{N}-\proj^{N-1})\|^2+\frac{\alpha}{2}\|\gamedio A(\prim^{1}-\proj^{0})\|^2\nonumber\\
   &+\frac{\alpha}{2}\sum_{k=1}^{N-1}\|\gamedio A(\prim^{k+1}-\proj^{k})\|^2\nonumber\\ \geq  &  \frac{\alpha}{4}\sum_{k=1}^{N-1}\|\gamedio ( A\prim^{k+1}-b^{\delta})\|^2-\frac{\|\prim^{1}-\proj^{0}\|_{\pSigma^{-1}}^2}{2}+\frac{\alpha}{2}\|\gamedio A(\prim^{1}-\proj^{0})\|^2\nonumber\\
   &+ \frac{\alpha^{2}+\|\gamedio A\pSigma^{\frac{1}{2}}\|^2}{2}\|\prim^{N}-\proj^{N-1}\|_{\pSigma^{-1}}^2\nonumber\\
   \geq& \frac{\alpha}{4}\sum_{k=2}^{N}\|\gamedio ( A\prim^{k}-b^{\delta})\|^2-\frac{\|\prim^{1}-\proj^{0}\|_{\pSigma^{-1}}^2}{2}+\frac{\alpha}{2}\|\gamedio A(\prim^{1}-\proj^{0})\|^2\nonumber\\
   &+ \frac{\alpha^{2}+\|\gamedio A\pSigma^{\frac{1}{2}}\|^2}{2}\|\prim^{N}-\proj^{N-1}\|_{\pSigma^{-1}}^2.\label{e: PX}
    \end{align}
On the other hand, we get
    \begin{align}
      \|\gamedio(A\prim^k-b^\delta)\|^2 \geq &\frac{\|A\prim^k-b^{\delta}\|^2}{\|\Gamma^{-1}\|}\nonumber\\\geq&\frac{1}{\|\Gamma^{-1}\|}\left(\frac{\|A\prim^k-b\|^2}{2}-\|b^\delta-b\|^2\right).
        \label{e: Axsep P}
    \end{align}
 Combining \eqref{e: lagrangeIII.5 P1}, 
          \eqref{e: lagrangeIII P}, \eqref{e: PX}, and \eqref{e: Axsep P} we have that
     \begin{align}
       &\sum_{k=0}^{N-1}\left(\mathcal{L}(\prim^{k+1},\dal)-\mathcal{L}(\prim,\dal^{k+1})\right)+\frac{\alpha^2}{2}\|\prim^{N}-\proj^{N-1}\|_{\pSigma^{-1}}^2\nonumber\\
       &\sum_{k=1}^{N}\frac{\alpha}{8\|\Gamma^{-1}\|}\|A\prim^{k+1}-b\|^2+\frac{\|\prim^{N}-\prim\|^2_{\pSigma^{-1}}}{2} \nonumber\\
       \leq& \delta\|\Gamma^{\frac{1}{2}}\|\sum_{k=1}^{N}\|\dal^{k}-\dal\|_{\Gamma^{-1}}+V(z^{0}-z)+\frac{N e\delta^2}{2}+N\frac{\alpha}{4\|\Gamma^{-1}\|}\delta^2
        \label{e: lagrangeIV P}
\end{align}
    It remains to bound $\delta\|\Gamma^{\frac{1}{2}}\|\sum_{k=1}^{N}\|\dal^{k}-\dal\|_{\Gamma^{-1}}$. From \eqref{e: lagrangeIV.5 P} and since $(x,u)$ is a saddle-point of the Lagrangian  we deduce that\\
    \begin{align}
     \|\dal^{N}-\dal\|^2\leq\frac{2\|\Gamma^{\frac{1}{2}}\|\delta}{\alpha}  \sum_{k=1}^{N}\|\dal^{k}-\dal\|+\frac{2 V(z^{0}-z)}{\alpha}+\frac{ N e\delta^2}{\alpha}.
        \label{e: U bound P}
    \end{align}
    Applying \cite[Lemma A.1]{rasch2020inexact} to Equation \eqref{e: U bound P} with $\lambda_{k}:=\frac{2\|\Gamma^{\frac{1}{2}}\|\delta}{\alpha} $ and $S_{N}:= \frac{2 V(z^{0}-z)}{\alpha}+\frac{N e\delta^2}{\alpha}$  we get
    \begin{align}
      \|\dal^{N}-\dal\|\leq& \frac{N\|\Gamma^{\frac{1}{2}}\|\delta}{\alpha}+\left(\frac{2 V(z^{0}-z)}{\alpha}+\frac{ N e\delta^2}{\alpha}+\left(\frac{N\|\Gamma^{\frac{1}{2}}\|\delta}{\alpha}\right)^2\right)^{\frac{1}{2}}\nonumber\\\leq& \frac{2N\|\Gamma^{\frac{1}{2}}\|\delta}{\alpha}+\left(\frac{2 V(z^{0}-z)}{\alpha}\right)^{\frac{1}{2}}+\left(\frac{ N e\delta^2}{\alpha}\right)^{\frac{1}{2}}
    \end{align}
    Insert the previous in Equation \eqref{e: lagrangeIV.5 P}, to obtain
     \begin{align}
        \sum_{k=0}^{N-1}\left(\mathcal{L}(\prim^{k+1},\dal)-\mathcal{L}(\prim,\dal^{k+1})\right)\leq& \frac{2(N\|\Gamma^{\frac{1}{2}}\|\delta)^{2}}{\alpha}+N\|\Gamma^{\frac{1}{2}}\|\delta\left(\frac{ V(z^{0}-z)}{\alpha}\right)^{\frac{1}{2}}\nonumber\\&+N\|\Gamma^{\frac{1}{2}}\|\delta\left(\frac{ N e\delta^2}{\alpha}\right)^{\frac{1}{2}} +V(z^{0}-z)+\frac{N e\delta^2}{2}
        \label{e: lagrangeV P}
    \end{align}
    Analogously from \eqref{e: lagrangeIV P}
    \begin{align}
         \sum_{k=1}^{N}\|A\prim^k-b\|^2\leq&\frac{16N^2\|\Gamma\|\|\Gamma^{-1}\|\delta^{2}}{\alpha^{2}}+8N\delta\|\Gamma^{-1}\|\left(\frac{2\|\Gamma\| V(z^{0}-z)}{\alpha^3}\right)^{\frac{1}{2}}+8N\delta^{2}\|\Gamma^{-1}\|\left(\frac{ \|\Gamma\|e N}{\alpha^3}\right)^{\frac{1}{2}\nonumber}\\
         &+\frac{8\|\Gamma^{-1}\|V(z^{0}-z)}{\alpha}+2N\delta^{2}+\frac{4N\|\Gamma^{-1}\|e\delta^2}{\alpha}
    \end{align}
    and both results are straightforward from the Jensen's inequality.\end{proof}
\subsection{Proof of Theorem \ref{Th:PDP}}\label{Proof:PDP}
\begin{proof}
It follows from  \ref{A: DPSP} that

\begin{align}
    \pSigma^{-1}(\prim^k-\prim^{k+1})-A^*\bar{\prop}^{k}\in\partial J(\prim^{k+1})\nonumber\\
  \Gamma^{-1}(\prop^k-\dal^{k+1})+A\prim^{k+1} =b^{\delta}
  \label{I: DPinclsuion}
\end{align}
Thus,\\
\begin{align}
    J(\prim^{k+1})+\scal{\pSigma^{-1}(\prim^k-\prim^{k+1})- A^*\bar{\prop}^{k}}{\prim-\prim^{k+1}}\leq J(\prim)
    \label{e:subdif D}
\end{align} and \eqref{e:subdif D} yields
\begin{align}
         0\geq& J(\prim^{k+1})-J(\prim)+\scal{\pSigma^{-1}(\prim^{k}-\prim^{k+1})-A^*\bar{\prop}^{k}}{\prim-\prim^{k+1}} \nonumber\\ =& J(\prim^{k+1})-J(\prim)+\frac{\|\prim^{k}-\prim^{k+1}\|_{\pSigma^{-1}}^2}{2}+\frac{\|\prim^{k+1}-\prim\|_{\pSigma^{-1}}^2}{2}\nonumber\\&-\frac{\|\prim^{k}-\prim\|_{\pSigma^{-1}}^2}{2}+\scal{\prim^{k+1}-\prim}{A^*\bar{\prop}^{k}} 
                 \label{e: psub D}
\end{align}
From \eqref{I: DPinclsuion}, it follows that
\begin{align}
    0=&\scal{\Gamma^{-1}(\prop^k-\dal^{k+1})+A\prim^{k+1}-b^\delta}{\dal-\dal^{k+1}}\nonumber\\  0=&\frac{\|\dal^{k+1}-\prop^{k}\|_{\Gamma^{-1}}^2}{2}+\frac{\|\dal^{k+1}-\dal\|_{\Gamma^{-1}}^2}{2}-\frac{\|\prop^{k}-\dal\|_{\Gamma^{-1}}^2}{2}+\scal{b^{\delta}-A\prim^{k+1}}{\dal^{k+1}-\dal}\label{e: dsub D}
\end{align}
Recall that $z:=(\prim,\dal)\in\mathcal{Z}\subset C\times \RR^{d}$, $z^{k}:=(\prim^{k},\dal^{k})$, and $V(z):=\frac{\|\prim\|^2_{\pSigma^{-1}}}{2}+\frac{\|\dal\|_{\Gamma^{-1}}^2}{2}$. Summing \eqref{e: psub D} and \eqref{e: dsub D}, we obtain 
\begin{align}
         J(\prim^{k+1})-J(\prim)+\frac{\|\prim^{k+1}-\prim^{k}\|_{\pSigma^{-1}}^2}{2} +\frac{\|\dal^{k+1}-\prop^{k}\|_{\Gamma^{-1}}^2}{2}+V(z^{k+1}-z)-V(z^{k}-z)&\nonumber\\+\scal{A(\prim^{k+1}-\prim)}{\bar{\prop}^{k}} +\scal{b^{\delta}-A\prim^{k+1}}{\dal^{k+1}-\dal}&\leq0\label{e: prestimate D}
    \end{align}
    Now compute 
     \begin{align}
        & J(\prim^{k+1})-J(\prim)+\scal{A(\prim^{k+1}-\prim)}{\bar{\prop}^{k}} +\scal{b^{\delta}-A\prim^{k+1}}{\dal^{k+1}-\dal}\nonumber\\ =&\mathcal{L}(\prim^{k+1},\dal)-\mathcal{L}(\prim,\dal^{k+1})-\scal{A\prim^{k+1}-b}{\dal}+\scal{A\prim-b}{\dal^{k+1}}\nonumber\\&+\scal{A(\prim^{k+1}-\prim)}{\bar{\prop}^{k}} +\scal{b^{\delta}-A\prim^{k+1}}{\dal^{k+1}-\dal}\nonumber\\ =&\mathcal{L}(\prim^{k+1},\dal)-\mathcal{L}(\prim,\dal^{k+1})-\scal{A\prim^{k+1}}{\dal}+\scal{b}{\dal}+\scal{A\prim}{\dal^{k+1}}-\scal{b}{\dal^{k+1}}\nonumber\\&+\scal{A(\prim^{k+1}-\prim)}{\bar{\prop}^{k}} +\scal{b^\delta}{\dal^{k+1}-\dal}-\scal{A\prim^{k+1}}{\dal^{k+1}}+\scal{A\prim^{k+1}}{\dal}\nonumber\\ =&\mathcal{L}(\prim^{k+1},\dal)-\mathcal{L}(\prim,\dal^{k+1})+\scal{b^{\delta}-b}{\dal^{k+1}-\dal}+\scal{A(\prim^{k+1}-\prim)}{\bar{\prop}^{k}-\dal^{k+1}}\label{e: lagrange Di1}\\ =&\mathcal{L}(\prim^{k+1},\dal)-\mathcal{L}(\prim,\dal^{k+1})+\scal{b^{\delta}-b}{\dal^{k+1}-\dal}+\scal{A(\prim^{k+1}-\prim)}{\prop^{k}-\dal^{k+1}}\nonumber\\ &+\scal{A(\prim^{k+1}-\prim)}{\dal^{k}-\prop^{k-1}}\nonumber\\ =&\mathcal{L}(\prim^{k+1},\dal)-\mathcal{L}(\prim,\dal^{k+1})+\scal{b^{\delta}-b}{\dal^{k+1}-\dal}+\scal{A(\prim^{k+1}-\prim)}{\prop^{k}-\dal^{k+1}}\nonumber\\ &+\scal{A(\prim^{k}-\prim)}{\dal^{k}-\prop^{k-1}}+\scal{A(\prim^{k+1}-\prim^{k})}{\dal^{k}-\prop^{k-1}}\nonumber\\ =&\mathcal{L}(\prim^{k+1},\dal)-\mathcal{L}(\prim,\dal^{k+1})+\scal{b^{\delta}-b}{\dal^{k+1}-\dal}+\scal{A(\prim^{k+1}-\prim)}{\prop^{k}-\dal^{k+1}}\nonumber\\ &+\scal{A(\prim^{k}-\prim)}{\dal^{k}-\prop^{k-1}}+\scal{\gamedio A(\prim^{k+1}-\prim^{k})}{\gamediomenos(\dal^{k}-\prop^{k-1})}.
        \label{e: lagrange D}
    \end{align}
   From \eqref{e: lagrange D} and \eqref{e: prestimate D} we obtain
\begin{align}
               &\mathcal{L}(\prim^{k+1},\dal)-\mathcal{L}(\prim,\dal^{k+1})+\frac{\|\prim^{k+1}-\prim^{k}\|_{\pSigma^{-1}}^2}{2} +\frac{\|\dal^{k+1}-\prop^{k}\|_{\Gamma^{-1}}^2}{2}+V(z^{k+1}-z)-V(z^{k}-z) \nonumber\\ \leq&-\scal{b^{\delta}-b}{\dal^{k+1}-\dal} -\scal{A(\prim^{k+1}-\prim)}{\prop^{k}-\dal^{k+1}} +\scal{A(\prim^{k}-\prim)}{\prop^{k-1}-\dal^{k}}\nonumber\\ &-\scal{\gamedio A(\prim^{k+1}-\prim^{k})}{\gamediomenos(\dal^{k}-\prop^{k-1})}\nonumber\\\leq&\delta\|\Gamma^{\frac{1}{2}}\|\|\dal^{k+1}-\dal\|_{\Gamma^{-1}}  -\scal{A(\prim^{k+1}-\prim)}{\prop^{k}-\dal^{k+1}} +\scal{A(\prim^{k}-\prim)}{\prop^{k-1}-\dal^{k}}\nonumber\\ &+\frac{\|\prim^{k+1}-\prim^{k}\|_{\pSigma^{-1}}^2}{2}+\|\gamedio A\pSigma^{\frac{1}{2}}\|^2\frac{\|\dal^{k}-\prop^{k-1}\|_{\Gamma^{-1}}^2}{2}\label{e: lagrangeI D}
    \end{align}
    Therefore we have that
\begin{align}
       &\mathcal{L}(\prim^{k+1},\dal)-\mathcal{L}(\prim,\dal^{k+1}) +\frac{\|\dal^{k+1}-\prop^{k}\|_{\Gamma^{-1}}^2}{2}-\|\gamedio A\pSigma^{\frac{1}{2}}\|^2\frac{\|\dal^{k}-\prop^{k-1}\|_{\Gamma^{-1}}^2}{2}\nonumber\\&+V(z^{k+1}-z)-V(z^{k}-z) \nonumber\\\leq&\delta\|\Gamma^{\frac{1}{2}}\|\|\dal^{k+1}-\dal\|_{\Gamma^{-1}}  -\scal{A(\prim^{k+1}-\prim)}{\prop^{k}-\dal^{k+1}} +\scal{A(\prim^{k}-\prim)}{\prop^{k-1}-\dal^{k}}
        \label{e: lagrangeII D}
    \end{align}
    Summing from $1$ to $N-1$ we obtain
    \begin{align}
         &\sum_{k=1}^{N-1}\left(\mathcal{L}(\prim^{k+1},\dal)-\mathcal{L}(\prim,\dal^{k+1}) \right)+\frac{\alpha}{2}\sum_{k=1}^{N-1}\|\dal^{k+1}-\prop^{k}\|_{\Gamma^{-1}}^2\nonumber\\&+V(z^{N}-z)+\|\gamedio A\pSigma^{\frac{1}{2}}\|^2\frac{\|\dal^{N}-\prop^{N-1}\|_{\Gamma^{-1}}^2}{2}\nonumber\\\leq&\delta\|\Gamma^{\frac{1}{2}}\|\sum_{k=1}^{N-1}\|\dal^{k+1}-\dal\|  -\scal{A(\prim^{N}-\prim)}{\prop^{N-1}-\dal^{N}}+\scal{A(\prim^{1}-\prim)}{\prop^{0}-\dal^{1}} +V(z^{1}-z) \nonumber\\\leq&\delta\|\Gamma^{\frac{1}{2}}\|\sum_{k=1}^{N-1}\|\dal^{k+1}-\dal\| +\|\gamedio A\pSigma^{\frac{1}{2}}\|^2\frac{\|\dal^{N}-\prop^{N-1}\|_{\Gamma^{-1}}^2}{2} +\frac{\|\prim^{N}-\prim\|_{\pSigma^{-1}}^2}{2} \nonumber\\&+\scal{A(\prim^{1}-\prim)}{\prop^{0}-\dal^{1}} +V(z^{1}-z)
        \label{e: lagrangeIII D}
    \end{align}
    Reordering \eqref{e: lagrangeIII D} we obtain
     \begin{align}
        &\sum_{k=1}^{N-1}\left(\mathcal{L}(\prim^{k+1},\dal)-\mathcal{L}(\prim,\dal^{k+1})\right)+\frac{\alpha}{2}\sum_{k=1}^{N-1}\|\dal^{k+1}-\prop^{k}\|_{\Gamma^{-1}}^2+\frac{\|\dal^{N}-\dal\|_{\Gamma^{-1}}^2}{2} \nonumber\\\leq& \delta\|\Gamma^{\frac{1}{2}}\|\sum_{k=1}^{N-1}\|\dal^{k+1}-\dal\|+\scal{A(\prim^{1}-\prim)}{\prop^{0}-\dal^{1}} +V(z^{1}-z).
        \label{e: lagrangeIV D}
    \end{align}
On the other hand, from \eqref{e: prestimate D}, \eqref{e: lagrange Di1}, and \eqref{e: Axsep D} we get
    \begin{align}
         \mathcal{L}(\prim^{1},\dal)-\mathcal{L}(\prim,\dal^{1})+\frac{\alpha}{2}\|\dal^{1}-\prop^{0}\|^2\leq& \delta\|\dal^{1}-\dal\|-\scal{A(\prim^{1}-\prim)}{\bar{\prop}^{0}-\dal^{1}}\nonumber\\ &V(z^{0}-z)-V(z^{1}-z)\label{e: firsstiteation}
    \end{align}
    Summing \eqref{e: lagrangeIV D} and \eqref{e: firsstiteation} yields
    \begin{align}
        &\sum_{k=1}^{N}\left(\mathcal{L}(\prim^{k},\dal)-\mathcal{L}(\prim,\dal^{k})\right)+\frac{\alpha}{2}\sum_{k=1}^{N}\|\dal^{k}-\prop^{k-1}\|_{\Gamma^{-1}}^2+\frac{\|\dal^{N}-\dal\|_{\Gamma^{-1}}^2}{2} \nonumber\\\leq& \delta\|\Gamma^{\frac{1}{2}}\|\sum_{k=1}^{N}\|\dal^{k}-\dal\|+V(z^{0}-z).
        \label{e: lagrangeVIII D}
    \end{align}
    Moreover, since $\dal^{k+1}-\prop^{k}=\Gamma(A\prim^{k+1}-b^{\delta})$
\begin{align}
       \|\dal^{k+1}-\proj^{k}\|_{\Gamma^{-1}}^2=&  \scal{\Gamma (A\prim^{k+1}-b^\delta)}{A\prim^{k+1}-b^\delta}\nonumber\\\geq&\frac{\|A\prim^{k+1}-b^\delta\|^2}{\|\Gamma^{-1}\|}\nonumber\\\geq&\frac{1}{\|\Gamma^{-1}\|}\left(\frac{\|A\prim^{k+1}-b\|^2}{2}-\|b^\delta-b\|^2\right)
        \label{e: Axsep D}
    \end{align}
     and from \eqref{e: lagrangeVIII D} and \eqref{e: Axsep D} we obtain
    \begin{align}
        &\sum_{k=0}^{N-1}\left(\mathcal{L}(\prim^{k+1},\dal)-\mathcal{L}(\prim,\dal^{k+1})\right)+\frac{\alpha}{4\|\Gamma^{-1}\|}\sum_{k=1}^{N}\|A\prim^{k}-b\|^2+\frac{\|\dal^{N}-\dal\|_{\Gamma^{-1}}^2}{2} \nonumber\\\leq& \delta\|\Gamma^{\frac{1}{2}}\|\sum_{k=1}^{N}\|\dal^{k}-\dal\|_{\Gamma^{-1}}+V(z^{0}-z)+\frac{\alpha N\delta^{2}}{2\|\Gamma^{-1}\|}
        \label{e: lagrangeV D}
    \end{align}
From \eqref{e: lagrangeVIII D} it follows that
    \begin{align}
     \|\dal^{k}-\dal\|_{\Gamma^{-1}}^2\leq2\delta\|\Gamma^{\frac{1}{2}}\|\sum_{k=1}^{N}\|\dal^{k}-\dal\|_{\Gamma^{-1}}+2V(z^{0}-z) ,
        \label{e: U bound D}
    \end{align}
    Apply \cite[Lemma A.1]{rasch2020inexact} to Equation \eqref{e: U bound D} with $\lambda_{k}:=2\delta\|\Gamma^{\frac{1}{2}}\|$ and $S_{k}:= 2V(z^{0}-z)$  to get
    \begin{align}
      \|\dal^{k}-\dal\|_{\Gamma^{-1}}\leq& N\|\Gamma^{\frac{1}{2}}\|\delta+\left(2 V(z^{0}-z)+\left(N\|\Gamma^{\frac{1}{2}}\|\delta\right)^2\right)^{\frac{1}{2}}\nonumber\\\leq& 2N\|\Gamma^{\frac{1}{2}}\|\delta+\left(2 V(z^{0}-z)\right)^{\frac{1}{2}}\label{e: ubound D}
    \end{align}
    Insert the previous in Equation \eqref{e: lagrangeVIII D}, to obtain
     \begin{align}
        \sum_{k=0}^{N-1}\left(\mathcal{L}(\prim^{k+1},\dal)-\mathcal{L}(\prim,\dal^{k+1})\right)\leq& 2\|\Gamma^{\frac{1}{2}}\|^{2} N^2\delta^{2}+N\|\Gamma^{\frac{1}{2}}\|\delta\left(2V(z^{0}-z)\right)^{\frac{1}{2}}+V(z^{0}-z)
        \label{e: lagrangeVI D}
    \end{align}
    and by \eqref{e: lagrangeV D} and  \eqref{e: ubound D} we have
 \begin{align}
        \sum_{k=1}^{N}\|A\prim^{k}-b\|^2\leq&\frac{4\|\Gamma^{-1}\|}{\alpha} &\left(2\|\Gamma^{\frac{1}{2}}\|^{2} N^2\delta^{2}+N\|\Gamma^{\frac{1}{2}}\|\delta\left(2V(z^{0}-z)\right)^{\frac{1}{2}}+V(z^{0}-z)+\frac{\alpha N\delta^{2}}{2\|\Gamma^{-1}\|}\right)
        \label{e: lagrangeVII D}
    \end{align}
     and both results follows from the Jensen's inequality.\end{proof}
     \subsection{Proof of Lemma \ref{L: Series Parallel}}\label{LP: Series Parallel}
     
\begin{proof} Note that every single equation in $C$ and $C_{\delta}$ 
    Let us first recall that
    \begin{align}
   P^{\delta}\hspace{2mm} \prim \mapsto    \prim+\frac{b^{\delta}_{j}-\scal{a_{j}}{\prim}}{\|a_{j}\|^2}a_{j}^{*}
\end{align}
     Note that the $j$-th equation of $C$ and $C_{\delta}$ are parallel. Then, for every $j\in [d]$ and $\bar{\prim}\in C$, we get 
\begin{align}
    \|P^{\delta}_{j}\prim-\bar{\prim}\|^2=&\|P_{j}\prim-\bar{\prim}\|^2+2\scal{P_{j}\prim-\bar{\prim}}{P^{\delta}_{j}\prim-P_{j}\prim}\nonumber\\&+\|P_{j}\prim-P^{\delta}_{j}\prim\|^2\nonumber\\=&\|P_{j}\prim-\bar{\prim}\|^2+\|P_{j}\prim-P^{\delta}_{j}\prim\|^2,
    \label{E:Pitagoras_1}
    \end{align}
   analogously, we have that
    \begin{align}
        \|\prim-\bar{\prim}\|^2=&\|\prim-P_{j}\prim\|^2+\|P_{j}\prim-\bar{\prim}\|^2.
        \label{E:Pitagoras_2}
    \end{align}
    It follows from \eqref{E:Pitagoras_1} and \eqref{E:Pitagoras_2} that
    \begin{align}
        \|P^{\delta}_{j}\prim-\bar{\prim}\|^2+\|\prim-P_{j}\prim\|^2=\|\prim-\bar{\prim}\|^2+\|P^{\delta}_{j}\prim-P_{j}\prim\|^2
   ,
    \end{align}
    hence
    \begin{align}
        \|P^{\delta}_{j}\prim-\bar{\prim}\|^2&\leq\|\prim-\bar{\prim}\|^2+\|P^{\delta}_{j}\prim-P_{j}\prim\|^2\nonumber\\ &\leq \|\prim-\bar{\prim}\|^2+\frac{(b^{\delta}_{j}-b_{j})^2}{\|a_{j}\|^2}\nonumber\\ &\leq \|\prim-\bar{\prim}\|^2+\frac{\delta^2}{\|a_{j}\|^2}\label{P: Paralellogram}
    \end{align}
    \begin{enumerate}
        \item  Since $T=P^{\delta}_{\beta_{l}}\circ\cdots\circ P^{\delta}_{\beta_{1}}$ it is clear that $C_{\delta}\subset\Fix T$ and  by induction we have that,
\begin{align}
\|T\prim-\bar{\prim}\|^2&\leq\|\prim-\bar{\prim}\|^2+e\delta^{2},
\end{align}
where $e=\frac{l}{\max\limits_{i=1,\dots,d}\|a_i\|}$.
\item The proof follows from the convexity of $\|\cdot\|^{2}$ which is obtained with $e=\frac{1}{\max\limits_{i=1,\dots,d}\|a_i\|}$.
\item Let $\bar{\prim}\in C$, by \eqref{d: averaged3}, we have
\begin{align}
    \|T\prim-\bar{\prim}\|^2=&\|\prim-\bar{\prim}\|^2-2\alpha\scal{\prim-\bar{\prim}}{A^{*}(A\prim-b^\delta)}+\alpha^2\|A^{*}(A\prim-b^\delta)\|^2\nonumber\\=&\|\prim-\bar{\prim}\|^2-2\alpha\scal{\prim-b}{A\prim-b^\delta}+\alpha^2\|A^{*}(A\prim-b^\delta)\|^2\nonumber\\ \leq&\|\prim-\bar{\prim}\|^2-2\alpha\scal{b^\delta-b}{A\prim-b^\delta}+\left(\alpha^2\|A\|^{2}-2\alpha\right)\|A\prim-b^{\delta}\|^2\label{e:Steepest descentalpha}
\end{align}
Now using the Young inequality with parameter $2-\alpha\|A\|^2$, we have that
\begin{align}
    \|T\prim-\bar{\prim}\|^2\leq&\|\prim-\bar{\prim}\|^2+\frac{\alpha}{2-\alpha\|A\|^2}\|b^{\delta}-b\|^2\nonumber\\\leq&\|\prim-\bar{\prim}\|^2+\frac{\alpha\delta^{2}}{2-\alpha\|A\|^2}.
\end{align} It remains to prove that if $C_{\delta}\neq0$ then $C_{\delta}\subset \Fix T$, which is clear from \eqref{d: averaged3}.
\item Let $\bar{\prim}\in C$ and $x\in\RR^p$,  if $A^{*}Ax=A^{*}b^{\delta}$ then \eqref{A: pitagoras error 1} immediately holds. Otherwise, we have
\begin{align}
    \|T\prim-\bar{\prim}\|^2=&\|\prim-\bar{\prim}\|^2-2\beta(x)\scal{\prim-\bar{\prim}}{A^{*}(A\prim-b^\delta)}+\beta(x)^2\|A^{*}(A\prim-b^\delta)\|^2\nonumber\\=&\|\prim-\bar{\prim}\|^2-2\beta(x)\scal{A\prim-b}{A\prim-b^\delta}+\beta(x)^2\|A^{*}(A\prim-b^\delta)\|^2\nonumber\\ =&\|\prim-\bar{\prim}\|^2-2\beta(x)\scal{b^{\delta}-b}{A\prim-b^\delta}-2\beta(x)\|A\prim-b^\delta\|^2\nonumber\\&+\beta(x)^2\|A^{*}(A\prim-b^\delta)\|^2\label{e:Steepest descentbeta}
\end{align}
Now using the Young inequality with parameter $2-\beta(x)\frac{\|A^{*}(A\prim-b^\delta)\|^2}{\|A\prim-b^\delta\|^2}$, we have that
\begin{align}
    \|T\prim-\bar{\prim}\|^2\leq&\|\prim-\bar{\prim}\|^2+\frac{\beta(x)}{2-\beta(x)\frac{\|A^{*}(A\prim-b^\delta)\|^2}{\|A\prim-b^\delta\|^2}}\|b^{\delta}-b\|^2\nonumber\\\leq&\|\prim-\bar{\prim}\|^2+M\delta^{2}.
\end{align} Finally, it is clear from \eqref{d: averaged4} that if $C_{\delta}\neq0$ then $C_{\delta}\subset \Fix T$.
\end{enumerate} 
\end{proof}
\printbibliography
\end{document}